\documentclass[11pt]{amsart}

\usepackage[paperheight=11in, 
    paperwidth=8.5in,
    outer=1in,
    inner=1in,
    bottom=1in,
    top=1in]{geometry}

\usepackage{amsmath, amscd, amsthm, amsfonts, amssymb, mathrsfs}
\usepackage{bbm}
\usepackage[all]{xy}
\usepackage{graphicx}
\usepackage{comment}
\usepackage{url}
\usepackage{color}
\usepackage{tikz}
\usepackage[sc]{mathpazo}

\usepackage[hidelinks]{hyperref} 
\hypersetup{
    colorlinks,
    citecolor=magenta,
    filecolor=magenta,
    linkcolor=blue,
    urlcolor=black
}

\usepackage[alphabetic]{amsrefs}

\theoremstyle{definition}
\newtheorem{cor}[subsection]{Corollary}
\newtheorem{lem}[subsection]{Lemma}
\newtheorem{prop}[subsection]{Proposition}
\newtheorem{thm}[subsection]{Theorem}
\newtheorem{defn}[subsection]{Definition}
\newtheorem{eg}[subsection]{Example}
\newtheorem{rem}[subsection]{Remark}

\newtheorem{conj}[subsection]{Conjecture}
\newtheorem{ques}[subsection]{Question}

\newcommand{\QQ}{\mathbb Q}
\newcommand{\RR}{\mathbb R}

\newcommand{\PP}{\mathbb P}
\newcommand{\ZZ}{\mathbb Z}

\title{Divisors on matroids and their volumes}
\author{Christopher Eur}
\date{}

\begin{document}

\maketitle

\begin{abstract}\vspace{-20pt}
The classical volume polynomial in algebraic geometry measures the degrees of ample (and nef) divisors on a smooth projective variety.  We introduce an analogous volume polynomial for matroids, and give a complete combinatorial formula.  For a realizable matroid, we thus obtain an explicit formula for the classical volume polynomial of the associated wonderful compactification.  We then introduce a new invariant called the shifted rank volume of a matroid as a particular specialization of its volume polynomial, and discuss its algebro-geometric and combinatorial properties in connection to graded linear series on blow-ups of projective spaces.
\end{abstract}

\section{Introduction}

The volume polynomial in classical algebraic geometry measures the self-intersection number of a nef divisor, or equivalently the volume of its Newton-Okounkov body.  Here we define analogously the volume polynomial $VP_M$ for a matroid $M$.   Let $M$ be a matroid on a ground set $E$ with lattice of flats $\mathscr L_M$, and denote $\overline{\mathscr L_M} := \mathscr L_M\setminus\{\emptyset, E\}$.  Recall the definition of the Chow ring of a matroid:

\begin{defn}
The \textbf{Chow ring} of a simple matroid $M$ is the graded ring
$$A^\bullet(M):=\frac{ \ZZ[x_F : F\in \overline{\mathscr L_M}] }{ \langle x_Fx_{F'} \ | \ F,F' \textnormal{ incomparable}\rangle + \langle \sum_{F\ni i}x_F - \sum_{G\ni j} x_G \ | \ i,j\in E\rangle }$$
\end{defn}

In analogy to Chow rings in algebraic geometry, we call elements of $A^1(M)$ \textbf{divisors} on a matroid $M$.  The ring $A^\bullet(M)_\RR$ satisfies Poincar\'e duality \cite[Theorem 6.19]{AHK18} with the degree map
$$\begin{array}{cl}
\deg_M: A^{\operatorname{rk}M-1}(M)_\RR \overset \sim \to \RR \quad & \textnormal{ where }\deg_M(x_{F_1}x_{F_2}\cdots x_{F_d}) = 1\\
& \textnormal{ for every maximal chain } F_1\subsetneq \cdots \subsetneq F_d \textnormal{ in }  \overline{\mathscr L_M}.
\end{array}$$
In other words, $A^\bullet(M)_\RR$ is an Artinian Gorenstein $\RR$-algebra, so that the Macaulay inverse system gives a well-defined cogenerator $VP_M$ of $A^\bullet(M)_\RR$.

\newtheorem*{defn:volpoldefn}{Definition \ref{volpoldefn}}
\begin{defn:volpoldefn}
Let $M$ be a matroid of rank $r = d+1$.  The \textbf{volume polynomial} $VP_M(\underline t) \in \RR[t_F : F\in \overline{\mathscr L_M}]$ is the cogenerator of $A^\bullet(M)_\RR$, where $VP_M(\underline t)$ is normalized so that the coefficient of any monomial $t_{F_1}t_{F_2}\cdots t_{F_d}$ corresponding to a maximal chain of flats in $\mathscr L_M$ is $d!$.
\end{defn:volpoldefn}

When $M$ is realizable, $VP_M$ agrees with the classical volume polynomial of the wonderful compactification of the complement of the associated hyperplane arrangement of $M$.  While $VP_M$ is initially defined purely algebraically, we prove a completely combinatorial formula for $VP_M$, which follows from our first main theorem on the intersection numbers of divisors on a matroid.

\newtheorem*{thm:main}{Theorem \ref{main}}
\begin{thm:main}
Let $M$ be a matroid of rank $r = d+1$ on a ground set $E$, and $\emptyset = F_0 \subsetneq F_1 \subsetneq \cdots \subsetneq F_k \subsetneq F_{k+1} = E$ a chain of flats in $\mathscr L_M$ of ranks $r_i := \operatorname{rk} F_i$, and $d_1, \ldots, d_k$ be positive integers such that $\sum_i d_i = d$.  Denote by $\widetilde d_i := \sum_{j=1}^i d_j$.  Then
$$\deg(x_{F_1}^{d_1}\cdots x_{F_k}^{d_k}) = (-1)^{d-k} \prod_{i = 1}^{k}{d_i -1 \choose \widetilde d_i - r_i}\mu^{\widetilde d_i - r_i}(M|F_{i+1}/F_i)$$
where $\mu^i(M')$ denotes the $i$-th unsigned coefficient of the reduced characteristic polynomial $\overline{\chi}_{M'}(t) = \mu^0(M') t^{\operatorname{rk} M'-1} - \mu^1(M') t^{\operatorname{rk}M'-2} + \cdots \pm \mu^{\operatorname{rk}M'-1}(M')$ of a matroid $M'$.
\end{thm:main}

\newtheorem*{cor:regdim}{Corollary \ref{volpol}}
\begin{cor:regdim}
 Let the notations be as above.  The coefficient of $t_{F_1}^{d_1}\cdots t_{F_k}^{d_k}$ in $VP_M(\underline t)$ is
$$(-1)^{d-k}{d \choose d_1, \ldots, d_k} \prod_{i = 1}^{k} {d_i - 1 \choose \widetilde d_i -r_i} \mu^{\widetilde d_i - r_i}(M|F_{i+1}/F_i).$$
\end{cor:regdim}

As a first application, we give an explicit formula for the volumes of generalized permutohedra, adding to the ones given by Postnikov in \cite{Pos09}.

\newtheorem*{prop:GPvol}{Proposition \ref{volGP}}
\begin{prop:GPvol}
Let $z_{(\cdot )}: [n]\supset I \mapsto z_I\in \RR$ be a submodular function on the boolean lattice of subsets of $[n] := \{1, \ldots, n\}$ such that $z_\emptyset = z_{[n]} = 0$.  Then the volume of the generalized permutohedron $P(\underline z) = \{ (\underline x) \in \RR^n \ | \ \sum_{i\in [n]} x_i = z_{[n]}, \ \sum_{i\in I} x_i \leq z_I \ \forall I\subset [n]\} $ is
$$(n-1)! \operatorname{Vol} P(\underline z) = \sum_{I_\bullet, \underline d} (-1)^{d-k} {d \choose d_1, \ldots, d_k}\prod_{i=1}^k {d_i -1 \choose \widetilde d_i - |I_i|} {|I_{i+1}| - |I_i| - 1 \choose \widetilde d_i - |I_i|}z_{I_i}$$
where the summation is over chains $\emptyset \subsetneq I_1 \subsetneq \cdots \subsetneq I_k\subsetneq I_{k+1} = [n]$ and $\underline d = (d_1, \ldots, d_k)$ such that $\sum d_i = n-1$ and $\widetilde d_j:= \sum_{i=1}^j d_i$.
\end{prop:GPvol}

It was not clear to the author how the two formulas are related.

\newtheorem*{ques:GPvolques}{Question \ref{GPvolques}}
\begin{ques:GPvolques}
Can one derive the formula in Proposition \ref{volGP} directly from \cite[Corollary 9.4]{Pos09}, or vice versa?
\end{ques:GPvolques}

Under the Macaulay inverse system, the Chow ring $A^\bullet(M)$ and its cogenerator $VP_M$ are equivalent algebraic objects, but the volume polynomial $VP_M$ lends itself more naturally as a function of a matroid than the Chow ring.  As an illustration, by considering $VP_M(\underline t)$ as a polynomial in $\RR[t_S : S\in 2^E]$, we show that the map $M \mapsto VP_M$ is a valuation under matroid polytope subdivisions.

\newtheorem*{prop:val}{Proposition \ref{val}}
\begin{prop:val}
$M\mapsto  VP_M$ is a (non-additive) matroid valuation in the sense of \cite{AFR10}.
\end{prop:val}

That the volume polynomial of a matroid behaves well with respect to the matroid polytope and that the matroid-minor Hopf monoid structure arises in the proof of the proposition above suggest there may be a generalization of Hodge theory of matroids to arbitrary Lie type.

\newtheorem*{ques:ADE}{Question \ref{ADE}}
\begin{ques:ADE}
Is there are Hodge theory of Coxeter matroids, generalizing the Hodge theory of matroids as described in \cite{AHK18}?  (For Coxeter matroids, see \cite{BGW03}).
\end{ques:ADE}

 Moreover, any natural specialization of $VP_M$ defines an invariant of a matroid.  In our case, we consider specializing $VP_M$ via the rank function of a matroid to obtain a new invariant.

\newtheorem*{defn:voldefn}{Definition \ref{voldefn}}
\begin{defn:voldefn}
For a matroid $M$, define its \textbf{shifted rank divisor} $D_M$ to be
$$D_M := \sum_{F\in \overline{\mathscr L_M}} (\operatorname{rk} F) x_F,$$
and define the \textbf{shifted rank volume} of $M$ to be the volume of its shifted rank  divisor:
$$\operatorname{shRVol}(M) := VP_M(t_F := \operatorname{rk} F) = \deg\Big(\sum_{F\in \overline{\mathscr L_M}} (\operatorname{rk} F)x_F\Big)^{\operatorname{rk} M-1}.$$
\end{defn:voldefn}

We remark that in a forthcoming paper \cite{Eur19}, the author discusses various divisors on a matroid, including what the author calls the rank divisor, to which the shifted rank divisor here is closely related.  The shifted rank of a matroid seems to be a genuinely new invariant, as it is unrelated to classical invariants such as the Tutte polynomial or the volume of the matroid polytope; see Remark \ref{volrel}.  For realizable matroids, the shifted rank volume of a matroid measures how close the matroid is to the uniform matroid.

\newtheorem*{thm:realmax}{Theorem \ref{realmax}}
\begin{thm:realmax}
Let $M$ be a realizable matroid of rank $r$ on $n$ elements.  Then 
$$\operatorname{shRVol}(M) \leq \operatorname{shRVol}(U_{r,n}) = n^{r-1} \quad \textnormal{ with equality iff $M = U_{r,n}$.}
$$
\end{thm:realmax}

The primary tool for the proof of Theorem \ref{realmax} is algebro-geometric in nature, relying on the realizability of the matroid.  In the forthcoming work \cite{Eur19} the author gives a proof Theorem \ref{realmax} without the realizability condition, but the proof is not a combinatorial reflection of the geometric one we give here.  Trying to apply similar method as in the proof of Theorem \ref{realmax} for the non-realizable matroids naturally leads to the following question.

\newtheorem*{ques:NObody}{Question \ref{NObody}}
\begin{ques:NObody}
Is there a naturally associated convex body of which this volume polynomial is measuring the volume?  In other words, is there a theory of Newton-Okounkov bodies for general matroids (a.k.a.\ linear tropical varieties)?
\end{ques:NObody}

\begin{rem}
Recent works on the Chow ring of matroids have led to the resolution of the long-standing conjecture of Rota on the log-concavity of the coefficients of chromatic polynomials, first proven for realizable matroids in \cite{Huh12} and \cite{HK12}, and for general matroids in \cite{AHK18}.  Among the key tools in \cite{Huh12} is the Teissier-Khovanskii inequality for intersection numbers of nef divisors, which can be understood as a phenomenon of convexity:  The Newton-Okounkov body $\Delta(D)$ of a nef divisor $D$ is a convex body whose volume is the self-intersection number of $D$, and its existence reduces the Teissier-Khovanskii inequality to the Brunn-Minkowski inequality for volumes of convex bodies.

For general matroids, a combinatorial version of Teissier-Khovanskii inequality is proven in \cite{AHK18} by establishing Hodge theory analogues for the Chow ring of a matroid without explicit use of convex bodies.  Noting that matroids can be considered as tropical linear varieties (\cite{AK06}), the results of \cite{AHK18} suggest an existence of an analogue of Newton-Okounkov bodies for tropical linear varieties, and perhaps tropical varieties in general.  Our results here can be seen as a first step towards in a direction.
\end{rem}

On the flip side of maximal volumes, we have the following conjecture on the minimal values.

\newtheorem*{conj:min}{Conjecture \ref{min}}
\begin{conj:min}
The minimum volume among simple matroids of rank $r$ on $n$ is achieved uniquely by the matroid $U_{r-2,r-2}\oplus U_{2,n-r+2}$, and its volume is $r^{r-2}((n-r+1)(r-1)+1)$.
\end{conj:min}

\subsection*{Structure of the paper} In section \S2, we review relevant definitions and results about Chow rings of matroids and volume polynomials in algebraic geometry.  In section \S3, we define the volume polynomial of a matroid and give a combinatorial formula.  First applications of the volume polynomial is discussed in section \S4, where we give another formula for the volumes of generalized permutohedra and show that the volume polynomial is a matroid valuation.  In section \S5, we define the shifted rank volume of a matroid as a particular specialization of the volume polynomial, and analyze some of its algebro-geometric and combinatorial properties.  Lastly, in section \S6, we feature some examples.

A Macaulay2 file \verb+volumeMatroid.m2+ implements many of notions here, along with computations for the examples provided.  It can be found at \url{https://math.berkeley.edu/~ceur/research.html}.

\proof[Notations] $|S|$ denotes the cardinality of a (finite) set $S$.  We use $\mathbbm k$ for a field, which we always assume algebraically closed, and a $\mathbbm k$-variety is an integral separated scheme over $\mathbbm k$.  A binomial coefficient ${n \choose m}$ is understood to be zero if $m<0$ or $m>n$.

\section{Preliminaries}

In this section, we set up relevant notations and review previous results about Chow rings of matroids, wonderful compactifications, volume polynomials, and cogenerators in the classical setting.

\subsection*{Wonderful compactifications and Chow rings of matroids}   We assume some familiarity with the basics of matroid theory, and point to  \cite{Oxl11}, \cite{Wel76}, or the collection \cite{Whi86} as a general reference for matroids.  For accounts tailored towards Chow ring of matroids, we recommend \cite{Bak18}, \cite{Kat16}, or \cite{Huh18}.

\medskip
Among various axiomatic systems for matroids, the flats description reproduced below for the convenience of the reader is the most relevant in our setting.

\begin{defn}\label{defn:flats}
A \textbf{matroid} $M = (E,\mathscr L_M)$ consists of a finite set $E$, called its ground set, and a collection $\mathscr L_M$ of subsets $E$, called the set of \textbf{flats} of $M$, satisfying
\begin{enumerate}
\item[(F1)] $E\in \mathscr L_M$,
\item[(F2)] $F_1, F_2 \in \mathscr L_M \implies F_1 \cap F_2 \in \mathscr L_M$, and
\item[(F3)] for each $F\in \mathscr L_M$, let $\mathscr L_M^{\gtrdot F}$ be the set of flats that cover $F$, i.e.\
$$\mathscr L_M^{\gtrdot F} := \{G \in \mathscr L_M \ | \ G\supsetneq F, \textnormal{ and } G\supsetneq G' \supsetneq F \Rightarrow G'\notin \mathscr L_M\}.$$
Then $\{G \setminus F \ | \ G\in \mathscr L_M^{\gtrdot F}\}$ partition $E\setminus F$. 
\end{enumerate}
\end{defn}

The axiom (F3) above is often called the \textbf{cover-partition axiom} for flats of a matroid.  We will always assume that a matroid $M$ is \textbf{loopless}; that is, $\emptyset \in \mathscr L_M$.  The flats of a matroid $M$ form a lattice by inclusion; we will denote this lattice also by $\mathscr L_M$ and denote $\overline{ \mathscr L_M} := \mathscr L_M \setminus \{\emptyset, E\}$.  Moreover, we set the following notations for matroids: A (loopless) matroid $M$ of rank $r = d+1$ on a ground set $E$ has
\begin{itemize}
\item $\operatorname{rk}_M$ its rank function,
\item $\mathfrak A^\bullet(M)$ the atoms, i.e.\ rank 1 flats, of the lattice $\mathscr L_M$,
\item $\chi_M$ (resp.\ $\overline\chi_M$) the (resp.\ reduced) characteristic polynomial, and
\item $M|S$ (resp.\ $M/S$) the restriction (resp.\ contraction) of $M$ to (resp.\ by) $S\subset E$.
\end{itemize}
Additionally, for $F, F'$ flats of $M$ such that $F\subset F'$, we denote by
$$[F,F'] := \{G \in \mathscr L_M \ | \ F \subseteq G \subseteq F'\} \quad\textnormal{and}\quad (F,F') := \{G\in \mathscr L_M \ | \ F \subsetneq G \subsetneq F'\}$$
the closed interval and open interval defined by $F,F'$ in $\mathscr L_M$.  We note the following fact (\cite[Proposition 3.3.8]{Oxl11})
$$\mathscr L_{M|F'/F} \simeq [F,F'].$$ 

\medskip
The data of lattice of flats of $M = (E,\mathscr L_M)$ can be represented by a polyhedral fan called its Bergman fan.  We prepare by setting the following notations:
\begin{itemize}
\item $\{e_i \ | \ i \in E\}$ the standard basis of $\ZZ^E$ and $\langle \cdot, \cdot \rangle$ the standard dot-product on $\ZZ^E$,
\item $N := \ZZ^E/\ZZ \mathbf 1$ be a lattice where $\mathbf 1$ denotes the all 1 vector,
\item $u_i$ the image of $e_i$ in $N$ for $i\in E$,
\item the dual lattice $N^\vee$ of $N$, which identified to a sublattice of $\ZZ^E$ via the inner product $\langle \cdot, \cdot \rangle$ as $N^\vee = \mathbf 1^\perp := \{m \in \ZZ^E \ | \ \langle m, \mathbf 1 \rangle = 0 \}$.
\end{itemize}

\begin{defn}
Let $M$ be a matroid of rank $r = d+1$ on a ground set $E$.  With the notations as above, the \textbf{Bergman fan} $\Sigma_M$ is the pure $d$-dimensional polyhedral fan in $N_\RR := N \otimes \RR$ that is comprised of cones$$\sigma_{\mathscr F} := \operatorname{Cone}(u_{F_1}, u_{F_2}, \cdots, u_{F_k}) \subset N_\RR$$
for each chain of flats $\mathscr F: \emptyset \subsetneq F_1\subsetneq \cdots\subsetneq F_k\subsetneq E$ in $\mathscr L_M$.
\end{defn}


\medskip
We now motivate the definition of the Chow ring of a matroid by sketching its algebro-geometric origin as the cohomology ring the wonderful compactification of a hyperplane arrangement.  Let $M$ be a loopless matroid on $E = \{0,1,\ldots, n\}$ of rank $r=d+1$ realizable over a field $\mathbbm k$, which we may assume to be algebraically closed.  A \textbf{realization} $\mathscr R(M)$ of $M$ consists of the following equivalent data of
\begin{itemize}
\item a list of vectors $E = \{v_0, \ldots, v_n\}$ spanning a $\mathbbm k$-vector space $V\simeq \mathbbm k^r$,
\item a surjection $\mathbbm k^{n+1} \twoheadrightarrow V$ where $e_i \mapsto v_i$, or
\item an injection $\PP V^* \hookrightarrow \PP_{\mathbbm k}^n$, obtained by projectivizing the dual of the surjection $\mathbbm k^{n+1} \twoheadrightarrow V$.
\end{itemize}

For a realization $\mathscr R(M)$ of $M$ with $\PP V^* \hookrightarrow \PP^n$, the intersections of coordinate hyperplanes of $\PP^{n}$ with $\PP V^*$ then define then the associated \textbf{hyperplane arrangement}
$$\mathcal A_{\mathscr R(M)} = \{ L_a \}_{a\in \mathfrak A^\bullet(M)}, \textnormal{ where } L_a:= \{f\in \PP V^* \ | \ f(v_i) = 0 \ \forall v_i \in a\}$$
Each $c$-codimensional linear subspace $L_F\subset \PP V^*$, obtained by intersections of hyperplanes in $\mathcal A_{\mathscr R(M)}$, correspond to a rank $c$ flat $F$ of $M$ by
$$L_F := \{f\in \PP V^* \ | \ f(v_i) = 0 \ \forall v_i\in F\}.$$
We denote by $C(\mathcal A_{\mathscr R(M)})$ the hyperplane arrangement complement $C(\mathcal A_{\mathscr R(M)}) := \PP V^* \setminus \bigcup \mathcal A_{\mathscr R(M)}$.

\begin{defn}\label{defn:wndcpt}
The \textbf{wonderful compactification} $X_{\mathscr R(M)}$ of the complement $C(\mathcal A_{\mathscr R(M)})$ is obtained by a series of blow-ups on $\PP V^*$ in the following way:  First blow-up the points $\{L_F\}_{\operatorname{rk}(F) = d}$, then blow-up the strict transforms of the lines $\{L_F\}_{\operatorname{rk}(F) = d-1}$, and continue until having blown-up strict transforms of $\{L_F\}_{\operatorname{rk}(F) = 1}$.  Let $\pi_M: X_{\mathscr R(M)} \to \PP V^*$ be the blow-down map.
\end{defn}

See \cite{DCP95} for the original construction of $X_{\mathscr R(M)}$ or \cite{Fei05} for a survey geared towards combinatorialists.  The boundary $X_{\mathscr R(M)}\setminus C(\mathcal A_{\mathscr R(M)})$ consists of the exceptional divisors $\widetilde{L_F}$.  The intersection theory of these boundary divisors and moreover the cohomology ring of $X_{\mathscr R(M)}$ are encoded in the matroid.  This was first observed in \cite{DCP95}, and led the authors of \cite{FY04} to define the following.

\begin{defn}
Let $M$ be a (loopless) matroid on a ground set $E$, and notate $N = \ZZ^E/\ZZ\mathbf1$, $N^\vee = \mathbf1^\perp$ as before.  The \textbf{Chow ring}\footnote{Chow rings of algebraic varieties initially take coefficients in $\ZZ$, but often one can tensor by $\RR$ to work with algebraic cycles modulo numerical equivalence \cite[Appendix C.3.4]{EH16}.  In this article, we will always assume that we are working with coefficients in $\RR$ for convenience.} $A^\bullet(M)$ of $M$ is defined as the Chow ring of the toric variety of the Bergman fan $\Sigma_M \subset N_\RR$.  Explicitly, it is the graded ring
$$A^\bullet(M):= \frac{ \RR[x_F : F\in \overline{\mathscr L_M}] }{\mathcal I_M + \mathcal J_M}$$
where the ideal $\mathcal I_M$ consists of quadrics by
$$\mathcal I_M := ( x_Fx_F' \ | \ \textnormal{$F,F'$ not comparable})$$
and the ideal $\mathcal J_M$ consists of linear relations by
$$\mathcal J_M := \Big(\sum_{F\in \overline{\mathscr L_M}}\langle m, u_F\rangle x_F \ | \ m \in N^\vee \Big).$$
We call elements of $A^1(M)$ \textbf{divisors} on $M$.
\end{defn}

We remark that a more familiar presentation of the linear relations
$$\mathcal J_M = \textstyle \big( \sum_{F\ni i} x_F - \sum_{G\ni j} x_G \ | \ i, j \in E \big)$$
for the Chow ring $A^\bullet(M)$ is recovered by considering the generating set $\{e_i - e_j\}_{i,j\in E}$ of $N^\vee$.

\medskip
When $M$ has a realization $\mathscr R(M)$, the variables $x_F$ correspond to the exceptional divisors $\widetilde {L_F}$, and the cohomology ring of $X_{\mathscr R(M)}$ is isomorphic to $A^\bullet(M)$ (\cite[Corollary 2]{FY04}).

\begin{rem}\label{tropcpt} Let $M$ be a realizable matroid with realization $\mathscr R(M)$.  The Bergman fan $\Sigma_M$ of a realizable matroid $M$ is the tropicalization of the (very affine) linear variety $C(\mathcal A_{\mathscr R(M)})$ (\cite{AK06}).  The wonderful compactification $X_{\mathscr R(M)}$ is then a tropical compactification where we take the closure of $C(\mathcal A_M)$ in the toric variety $X_{\Sigma_M}$ of the Bergman fan.  The content of \cite[Corollary 2]{FY04} is that the inclusion $X_{\mathscr R(M)} \hookrightarrow X_{\Sigma_M}$ induces a Chow equivalence $A^\bullet(X_{\mathscr R(M)}) \simeq A^\bullet(X_{\Sigma_M})$, as the Chow ring of $X_{\Sigma_M}$, following \cite{Dan78} and \cite{Bri96}, satisfy $A^\bullet(X_{\Sigma_M}) \simeq A^\bullet(M)$.  For reference on tropical compactifications see \cite{MS15} or \cite{Den14}.
\end{rem}

Recently in \cite{AHK18}, the ring $A^\bullet(M)$ has been shown to satisfy the whole K\"ahler package---Poincar\'e duality, hard Lefschetz property, and Hodge-Riemann relations---which led to the proof of Rota's conjecture on the log-concavity of the unsigned coefficients of the characteristic polynomial of a matroid in the same paper.  For our purposes, we only need the Poincar\'e duality.

\begin{prop}\cite[Theorem 6.19]{AHK18} The Chow ring $A^\bullet(M)$ of a matroid $M$ of rank $r = d+1$ is a finite graded $\RR$-algebra satisfying the following.
\begin{enumerate}
\item There exists a linear isomorphism $\deg_M: A^d(M) \to \RR$ uniquely determined by the property that $\deg_M(x_{F_1}x_{F_2}\cdots x_{F_d}) = 1$ for every maximal chain $F_1\subsetneq \cdots \subsetneq F_d$ in $\overline{\mathscr L_M}$, and
\item For each $0\leq i \leq d$, the pairing $A^i(M) \times A^{d-i}(M) \to A^d(M) \overset{\deg}{\simeq}\RR$ is non-degenerate.
\end{enumerate}
\end{prop}

\subsection*{Volumes of divisors and cogenerators} For general reference on intersection theory, see \cite{EH16} and \cite{Laz04}.  Here we mostly follow the survey \cite{ELM+05}.

\medskip

Let $X$ be a $d$-dimensional smooth projective variety over an algebraically closed field $\mathbbm k$, and let $A^\bullet(X)$ be its Chow ring and denote by $\deg_X$ or $\int_X$ the degree map $A^d(X) \to \ZZ$ sending a class of a closed point to 1.  For a Cartier divisor $D$ on $X$, the \textbf{volume} of $D$ is defined as
$$\operatorname{vol}(D) := \lim_{t\to \infty} \frac{h^0(X,\mathscr O(tD))}{t^d/d!}.$$
In other words, denoting by $R(D)_\bullet := \bigoplus_{t\geq 0} H^0(X,tD)$ the \textbf{section ring} of $D$, the volume measures the asymptotics of $\frac{\dim_{\mathbbm k} R(D)_t}{t^d/d!}$ as $t\to \infty$.

If $D$ is ample, then $\operatorname{vol}(D)>0$.  By standard relation between Hilbert polynomials and intersection multiplicities, volume of an ample divisor can be geometrically interpreted as follows: If $m>>0$ is such that $mD$ is very ample, then for general divisors $E_1, \ldots, E_d$ in the complete linear system $|mD|$ we have $\operatorname{vol}(D) = \frac{1}{m^d} \deg_X[E_1\cap E_2 \cap \cdots \cap E_d]$.  In other words, $\operatorname{vol}(D) = \int_X (c_1(D))^d$ if $D$ is ample.  

The volume of a divisor depends only on its numerical equivalence class.  Thus, letting $N^1(X)$ be the group of divisors modulo numerical equivalence generated by $\{\xi_1, \ldots, \xi_r\}$ and $\operatorname{Nef}(X)\subset N^1(X)_\RR$ the nef cone, the map $\operatorname{vol}: \operatorname{Nef}(X) \to \RR$ defines the \textbf{volume polynomial} $VP_X\in \RR[t_1, \ldots, t_r]$ where 
$$VP_X(t_1, \ldots, t_r) = \operatorname{vol}(t_1\xi_1+\cdots + t_r\xi_r)$$
whenever $t_1\xi_1+ \cdots + t_r\xi_r\in \operatorname{Nef}(X)$.

\medskip
Macaulay's inverse system and more generally the Matlis duality provide a purely algebraic approach to the notion of volume polynomial as the dual socle generator of an Artinian Gorenstein ring.  We sketch the connection here and refer to \cite{BH93} for details.

For a graded (or local complete) Noetherian ring $S$ with the residue field $k$ and its injective hull $E(k)$, the Matlis duality establishes a bijection
$$\{\textnormal{Artinian ideals $I \subset S$}\} \longleftrightarrow \{\textnormal{Noetherian $S$-submodules of $E(k)$}\}$$
via $I \mapsto \operatorname{Hom}_S(S/I,E(k)) = (0:_E I)$.  When $S/I$ is Artinian and Cohen-Macaulay, the bijection interchanges the type $r(S/I) := \dim_k \operatorname{Hom}_S(k,S/I)$ and the minimal number of generators $\mu((0:_E I))$.  In particular, if $S/I$ is Artinian Gorenstein, then $(0:_E I)$ is generated by a single element called the \textbf{cogenerator} or \textbf{dual socle generator} of $I$.  

When $\operatorname{char} k = 0$ and $S = k[x_0, \ldots, x_n]$ the standard graded polynomial ring, its injective hull is $E = k[\partial_0, \ldots, \partial_n]$ where $S$ acts on $E$ by $f\cdot \partial_i := \frac{\partial f}{\partial x_i}$.  The following proposition then shows the equivalence of the cogenerator and the volume polynomial when the Chow ring $A(X)_\QQ$ is an Artinian Gorenstein ring.

\begin{prop}\cite[Theorem 13.4.7]{CLS11} \label{cogen}
Suppose a graded finite $k$-algebra $A = \bigoplus_{i=0}^d A_i$ satisfies the following:
\begin{enumerate}
\item[(i)] $A$ is generated in $A_1$, with $A_0 = k$,
\item[(ii)] there exists a $k$-linear isomorphism $\deg: A_d \to k$, and
\item[(iii)] $A_i \times A_{d-i} \to A_d \overset{deg}\simeq k$ is a non-degenerate pairing.
\end{enumerate}
Let $x_1, \ldots, x_n$ generate $A_1$, so that $A \simeq k[\underline x]/I$ for some ideal $I$. Then there exists $P\in k[t_1, \ldots, t_n]$ such that $$I = \{f\in k[\underline x] \ | \ f(\partial{t_1}, \ldots, \partial{t_n})\cdot P = 0\}.$$
Up to scaling by an element of $k$, this cogenerator is $\deg\big( (t_1x_1 + \cdots + t_nx_n)^d\big)$ (where we extend $\deg: A_d\to k$ to $A_d[t_1, \ldots, t_n] \to k[t_1, \ldots, t_n]$).
\end{prop}

\begin{rem}
In toric geometry, the volume of an ample divisor is realized as a volume of a rational convex polytope; for details see \cite[\S9,\S13]{CLS11}.  Moreover, recently this phenomenon of realizing the volume of a divisor as a volume of convex body was extended to arbitrary smooth complete varieties where the Newton-Okounkov bodies take the place of the rational convex polytopes.  See \cite{LM09} for a more geometric perspective with applications to big divisors and N\'eron-Severi groups, and \cite{KK12} for an approach using semigroups and with a view towards Alexandrov-Fenchel inequality and generalized Kushnirenko-Bernstein theorem.  For an application with more representation theoretic flavor, see \cite{Kav11} on volume polynomials and cohomology rings of spherical varieties.
\end{rem}

\section{The volume polynomial of a matroid}

As the Chow ring $A^\bullet(M)$ of a matroid $M$ satisfies Poincar\'e duality, that is, the conditions of Proposition \ref{cogen}, it is an Artinian Gorenstein algebra with a cogenerator, well-defined up to scaling by a nonzero element of $\RR$.

\begin{defn}\label{volpoldefn}
Let $M$ be a matroid of rank $r = d+1$.  The \textbf{volume polynomial} $VP_M \in \RR[t_F : F\in \overline{\mathscr L_M}]$ is the cogenerator of $A^\bullet(M)$, where $VP_M$ is normalized so that the coefficient of any monomial $t_{F_1}t_{F_2}\cdots t_{F_d}$ corresponding to a maximal chain of flats in $\mathscr L_M$ is $d!$.  Equivalently, via Proposition \ref{cogen} the volume polynomial is
$$VP_M = \deg_M \Big(\sum_{F\in \overline{\mathscr L_M}} x_F t_F\Big)^d$$
where $\deg_M:A^d(M)\to \RR$ is extended to $A^d[t_F \ | \ F\in \overline{\mathscr L_M}]\to \RR[t_F \ | \ F\in \overline{\mathscr L_M}]$.
\end{defn}

The coefficient of $t_{F_1}^{d_1}\cdots t_{F_k}^{d_k}$ for $d_1 + \cdots +d_k = d = \operatorname{rk} M -1$ in the volume polynomial $VP_M$ is ${d \choose d_1, \ldots, d_k}\deg_M(x_{F_1}^{d_1}\cdots x_{F_k}^{d_k})$.  Thus, the knowing the volume polynomial amounts to knowing all the intersection numbers $\deg_M(x_{F_1}^{d_1}\cdots x_{F_k}^{d_k})$.  The quadratic relations $\mathcal I_M$ for $A^\bullet(M)$ imply that if $F_1, \ldots, F_k$ don't form a chain then $x_{F_1}^{d_1} \cdots x_{F_k}^{d_k} =0$, and so it suffices to consider the case when $F_1 \subsetneq \cdots \subsetneq F_k$ is a chain in $\overline{\mathscr L_M}$.  The main theorem in this section is the combinatorial formula for all the intersection numbers.

\begin{thm}\label{main} Let $M$ be a matroid of rank $r = d+1$ on a ground set $E$.  Let $\emptyset = F_0 \subsetneq F_1 \subsetneq \cdots \subsetneq F_k \subsetneq F_{k+1} = E$ be a chain of flats in $\mathscr L_M$ with ranks $r_i := \operatorname{rk} F_i$, and let $d_1, \ldots, d_k$ be positive integers such that $\sum_i d_i = d$.  Denote by $\widetilde d_i := \sum_{j=1}^i d_j$.  Then
$$\deg(x_{F_1}^{d_1}\cdots x_{F_k}^{d_k}) = (-1)^{d-k} \prod_{i = 1}^{k}{d_i -1 \choose \widetilde d_i - r_i}\mu^{\widetilde d_i - r_i}(M|F_{i+1}/F_i)$$
where $\mu^i(M')$ denotes the $i$-th unsigned coefficient of the reduced characteristic polynomial $\overline{\chi}_{M'}(t) = \mu^0(M') t^{\operatorname{rk} M'-1} - \mu^1(M') t^{\operatorname{rk}M'-2} + \cdots + (-1)^{\operatorname{rk}M' -1}\mu^{\operatorname{rk}M'-1}(M')$ of a matroid $M'$.
\end{thm}

As an immediate corollary, we obtain a combinatorial formula for the volume polynomial.

\begin{cor}\label{volpol} Let the notations be as above.  The coefficient of $t_{F_1}^{d_1}\cdots t_{F_k}^{d_k}$ in $VP_M(\underline t)$ is
$$(-1)^{d-k}{d \choose d_1, \ldots, d_k} \prod_{i = 1}^{k} {d_i - 1 \choose \widetilde d_i -r_i} \mu^{\widetilde d_i - r_i}(M|F_{i+1}/F_i).$$
\end{cor}

\begin{rem}[$\overline{\mathcal M_{0,n}}$] When $M$ has a realization $\mathscr R(M)$, the volume polynomial $VP_M$ agrees with the classical volume polynomial of the wonderful compactification $X_{\mathscr R(M)}$.  In particular, when $M = M(K_{n-1})$ the matroid of the complete graph on $n-1$ vertices, we obtain the Deligne-Mumford space $\overline{\mathcal M_{0,n}}$ of rational curves with $n$ marked points (\cite[Theorem 6.4.12]{MS15}).  The numerical cones of $\overline{\mathcal M_{0,n}}$ are known to be complicated, as $\overline{\mathcal M_{0,n}}$ is not a Mori dream space in general (\cite{GK16}).  Nevertheless, our combinatorial formula for the volume polynomial allows for computation of the volume of any divisor in its ample cone.
\end{rem}

\begin{rem}[Computation] 
Computing $VP_M$ via Proposition \ref{cogen} alone, or more generally computing the intersection numbers via Gr\"obner bases quickly becomes infeasible as the matroid becomes larger.  For example, $M(K_6)$ (the matroid of the complete graph on 6 vertices) has 203 flats, so that the Chow ring $A^\bullet(M)$ has 201 variables.
\end{rem}

The rest of this section is dedicated to proving Theorem \ref{main}.  For the remainder of the section, we fix the following notation.
\begin{itemize}
\item $M$ is a matroid of rank $d = r+1$ on a ground set $E$ with $\Sigma_M\subset N_\RR = (\ZZ^E/\ZZ\mathbf1)_\RR$.
\item For a chain of flats $\mathscr F: F_1 \subsetneq \cdots \subsetneq F_k$ in $\overline{\mathscr L_M}$, we always set $F_0 = \emptyset$ and $F_{k+1} = E$.
\end{itemize}

For a chain $\mathscr F: F_1 \subsetneq \cdots \subsetneq F_k$ in $\overline{\mathscr L_M}$ and $(d_1, \ldots, d_k) \in \ZZ^k_{>0}$ with $\sum_{i=1}^k d_i = d$, our goal is to expand $x_{F_1}^{d_1} \cdots x_{F_k}^{d_k} \in A^d(M)$ into a sum of square-free monomials.  We do this broadly in three steps:
\begin{enumerate}
\item Whenever $d_i > 1$ for some $1\leq i \leq k$, Proposition \ref{prop:topple} provides a way to expand $x_{F_1}^{d_1} \cdots x_{F_k}^{d_k}$ as a sum of $x_Gx_{F_1}^{d_1} \cdots x_{F_i}^{d_i-1} \cdots x_{F_k}^{d_k}$ where $G \in (F_{i-1},F_i)$ or $G\in (F_i, F_{i+1})$, and if no such $G\in \overline{\mathscr L_M}$ exists then $x_{F_1}^{d_1}\cdots x_{F_k}^{d_k} = 0$.
\item Then, Proposition \ref{prop:support} informs us which monomials that show up as we iteratively apply Proposition \ref{prop:topple} evaluate to zero under $\deg_M$. We use this to obtain a well-controlled process for expanding out a monomial $x_{F_1}^{d_1}\cdots x_{F_k}^{d_k}$ into a sum of square-free monomials.
\item Lastly, in the process of the expansion, the monomials that appear pick up various coefficients, and Proposition \ref{prop:gamma} relates these coefficients to reduced characteristic polynomials of minors of $M$.
\end{enumerate}

\medskip
We prepare Proposition \ref{prop:topple} by a noting a choice of $m\in N^\vee = \mathbf 1^\perp$ that define a linear relation $\sum_{F\in \overline{\mathscr L_M}} \langle m, u_F\rangle x_F \in \mathcal J_M$ well-suited for expanding out $x_{F_1}^{d_1} \cdots x_{F_k}^{d_k}$.

\begin{lem}\label{lem:dualpoint}
For a chain $\mathscr F: \emptyset = F_0 \subsetneq F_1 \subsetneq \cdots \subsetneq F_k \subsetneq F_{k+1} = E$ in $\mathscr L_M$, define
$$m(\mathscr F, i) := - |F_{i+1}\setminus F_i|e_{F_i \setminus F_{i-1}}+|F_i\setminus F_{i-1}|e_{F_{i+1}\setminus F_i} \in N^\vee$$
for each $i = 1, \ldots, k$.  Then, for $F_j \in \mathscr F$, we have
$$
\langle m(\mathscr F, i), u_{F_j}\rangle =
\begin{cases}
0 & \textnormal{if $j\neq i$}\\
-|F_{i+1}\setminus F_i|\cdot|F_i\setminus F_{i-1}|& \textnormal{if $j = i$}
\end{cases}.
$$
\end{lem}

\begin{proof}
Noting that $\langle e_S, u_{S'}\rangle = |S\cap S'|$ for any $S,S'\subset E$, we have that 
\begin{itemize}
\item[] if $j \leq i-1$, then $\langle m(\mathscr F, i), u_{F_j}\rangle = -0 + 0 = 0$, and
\item[] if $j\geq i+1$, then $\langle m(\mathscr F, i), u_{F_j}\rangle  = - |F_{i+1}\setminus F_i|\cdot|F_i \setminus F_{i-1}|  +  |F_i\setminus F_{i-1}|\cdot|F_{i+1}\setminus F_i| = 0$.
\end{itemize} 
   Lastly, when $i = j$, we get $-|F_i\setminus F_{i-1}|\cdot|F_{i+1}\setminus F_i|$.
\end{proof}

As a consequence, we obtain our first key proposition.

\begin{prop}\label{prop:topple}
Let $\mathscr F: F_1\subsetneq \cdots \subsetneq F_k$ be a chain in $\overline{\mathscr L_M}$, and suppose we have $(d_1, \ldots, d_k)\in \ZZ^k_{>0}$ with $\sum_{i=1}^k d_i \leq d$ and $d_i > 1$ for some $i\in \{1, \ldots, k\}$.  Then we have
\begin{multline}\label{eq:topple}
x_{F_1}^{d_1} \cdots x_{F_i}^{d_i} \cdots x_{F_k}^{d_k} = \displaystyle \sum_{G\in (F_{i-1},F_i)} -\frac{|G\setminus F_{i-1}|}{|F_i\setminus F_{i-1}|}x_G\cdot x_{F_1}^{d_1} \cdots x_{F_i}^{d_i -1} \cdots x_{F_k}^{d_k}\\
  + \sum_{G\in (F_i,F_{i+1})} -\frac{|F_{i+1}\setminus G|}{|F_{i+1}\setminus F_i|} x_G\cdot x_{F_1}^{d_1} \cdots x_{F_i}^{d_i -1} \cdots x_{F_k}^{d_k}
\end{multline}
as elements in $A^\bullet(M)$.
\end{prop}

\begin{proof}
By Lemma \ref{lem:dualpoint}, the lattice point $m(\mathscr F,i)\in N^\vee$ defines the following element of $\mathcal J_M$:
$$\sum_{G\in \overline{\mathcal L_M}} \langle m(\mathscr F,i), u_G\rangle x_G =  (-|F_i\setminus F_{i-1}|\cdot|F_{i+1}\setminus F_i|)x_{F_i} + \sum_{G\notin \mathscr F} \langle m(\mathscr F,i),u_G\rangle x_G.$$
Dividing by $|F_i\setminus F_{i-1}|\cdot|F_{i+1}\setminus F_i|$, we obtain the following linear relation for elements in $A^\bullet(M)$.
\begin{align*}
x_{F_i} &= \displaystyle \sum_{G\notin \mathscr F} \frac{\langle m(\mathscr F, i), u_G\rangle}{|F_i\setminus F_{i-1}|\cdot|F_{i+1}\setminus F_i|} x_G\\[5mm]
&= \displaystyle \sum_{G\notin \mathscr F} \left( -\frac{|G\cap (F_{i}\setminus F_{i-1})|}{|F_{i}\setminus F_{i-1}|} + \frac{|G\cap (F_{i+1}\setminus F_i)|}{|F_{i+1}\setminus F_i|} \right) x_G,
\end{align*}
and hence, we have
\begin{equation}\label{eq:eq1}
x_{F_1}^{d_1} \cdots x_{F_k}^{d_k} = \sum_{G\notin \mathscr F} \left( -\frac{|G\cap (F_{i}\setminus F_{i-1})|}{|F_{i}\setminus F_{i-1}|} + \frac{|G\cap (F_{i+1}\setminus F_i)|}{|F_{i+1}\setminus F_i|} \right) x_G\cdot x_{F_1}^{d_1} \cdots x_{F_i}^{d_i -1} \cdots x_{F_k}^{d_k}.
\end{equation}
Note that $d_1, \ldots, d_{i-1}, d_i - 1, d_{i+1}, \ldots, d_k$ are all positive, and recall that $x_Fx_{F'} = 0$ in $A^\bullet(M)$ if $F,F'$ are incomparable.  Thus, in the summation \eqref{eq:eq1} above, we only need keep the $G\notin \mathscr F$ such that $F_j \subsetneq G \subsetneq F_{j+1}$ for some $j = 0, \ldots, k$.  Such $G$ is either in $(\emptyset, F_{i-1})$, $(F_{i-1}, F_i)$, $(F_i, F_{i+1})$, or $(F_{i+1}, E)$.  For each case we compute that
$$
\left( -\frac{|G\cap (F_{i}\setminus F_{i-1})|}{|F_{i}\setminus F_{i-1}|} + \frac{|G\cap (F_{i+1}\setminus F_i)|}{|F_{i+1}\setminus F_i|} \right) =
\begin{cases}
-0 + 0 = 0 & \textnormal{if $G\in (\emptyset, F_{i-1})$}\\[2mm]
-\frac{|G\setminus F_{i-1}|}{|F_i\setminus F_{i-1}|} + 0 = -\frac{|G\setminus F_{i-1}|}{|F_i\setminus F_{i-1}|} & \textnormal{if $G\in (F_{i-1}, F_i)$}\\[2mm]
-1 + \frac{|G\setminus F_i|}{|F_{i+1}\setminus F_i|} = -\frac{|F_{i+1}\setminus G|}{|F_{i+1}\setminus F_i|} & \textnormal{if $G\in (F_i, F_{i+1})$}\\[2mm]
-1 + 1 = 0 & \textnormal{if $G\in (F_{i+1}, E)$}
\end{cases}.
$$
Thus, we obtain the desired equality
\begin{multline*}
x_{F_1}^{d_1} \cdots x_{F_i}^{d_i} \cdots x_{F_k}^{d_k} = \displaystyle \sum_{G\in (F_{i-1},F_i)} -\frac{|G\setminus F_{i-1}|}{|F_i\setminus F_{i-1}|}x_G\cdot x_{F_1}^{d_1} \cdots x_{F_i}^{d_i -1} \cdots x_{F_k}^{d_k}\\
  + \sum_{G\in (F_i,F_{i+1})} -\frac{|F_{i+1}\setminus G|}{|F_{i+1}\setminus F_i|} x_G\cdot x_{F_1}^{d_1} \cdots x_{F_i}^{d_i -1} \cdots x_{F_k}^{d_k}.
\end{multline*}
\end{proof}

Proposition \ref{prop:topple} motivates the following operation on $\RR[x_F \ | \ F\in \overline{\mathscr L_M}]$.

\begin{defn}
For $F\in \overline{\mathscr L_M}$, the \textbf{toppling operation} $\mathcal T_F: \RR[x_F \ | \ F\in \overline{\mathscr L_M}] \to \RR[x_F \ | \ F \in \overline{\mathscr L_M}]$ \textbf{associated to $F$} is defined by its values on the monomials as follows.

\smallskip
$\mathcal T_F(x_{F_1}^{d_1}\cdots x_{F_{k}}^{d_{k}}) :=$
$$
\begin{cases}
0  \quad & \textnormal{if $F_1, \ldots, F_k$ do not form a chain}\\[5mm]
 \begin{split}
\textstyle\underset{G\in (F_{i-1},F_i)}\sum -\frac{|G\setminus F_{i-1}|}{|F_i\setminus F_{i-1}|}x_G\cdot x_{F_1}^{d_1} \cdots x_{F_i}^{d_i -1} \cdots x_{F_k}^{d_k} \\
\textstyle+ \underset{G\in (F_i,F_{i+1})}\sum  -\frac{|F_{i+1}\setminus G|}{|F_{i+1}\setminus F_i|} x_G\cdot x_{F_1}^{d_1} \cdots x_{F_i}^{d_i -1} \cdots x_{F_k}^{d_k}
\end{split} & \begin{array}{l}\textnormal{if $F_1\subsetneq \cdots \subsetneq F_k$ (relabeling if necessary)}\\
\textnormal{and $F_i = F$ with $d_i > 1$}\end{array}\\[9mm]
x_{F_1}^{d_1}\cdots x_{F_{k}}^{d_{k}} & \textnormal{otherwise}.
\end{cases}
$$
\end{defn}

Proposition \ref{prop:topple} implies that $\mathcal T_F(f) = f$  as elements in $A^\bullet(M)$ for any $f\in \RR[x_F \ | \ F\in \overline{\mathscr L_M}]$, and moreover, it implies that the monomials appearing in
$$(\mathcal T_{F_k}^{d_k - 1} \circ \cdots \circ \mathcal T_{F_1}^{d_1 - 1})(x_{F_1}^{d_1} \cdots x_{F_k}^{d_k})$$
with nonzero coefficients are square-free.  Our goal now is to compute $(\mathcal T_{F_k}^{d_k - 1} \circ \cdots \circ\mathcal T_{F_1}^{d_1 - 1})(x_{F_1}^{d_1} \cdots x_{F_k}^{d_k})$ for a chain $\mathscr F: F_1 \subsetneq \cdots \subsetneq F_k$ in $\overline{\mathscr L_M}$ and $(d_1, \ldots, d_k) \in \ZZ^k_{>0}$ satisfying $d_1 + \cdots + d_k = d$.  This expansion by iterated toppling operations is rather unwieldy; the following second key proposition provides the tool to make computation of $(\mathcal T_{F_k}^{d_k - 1} \circ \cdots\circ \mathcal T_{F_1}^{d_1 - 1})(x_{F_1}^{d_1} \cdots x_{F_k}^{d_k})$ more manageable.

\begin{prop}\label{prop:support}
Let $F_1 \subsetneq \cdots \subsetneq F_k$ be a chain in $\overline{\mathscr L_M}$, and let $(d_1, \ldots, d_k)\in \ZZ^k_{>0}$ satisfy $d_1 + \cdots + d_k = d$.  
\begin{enumerate}
\item[(a)]\label{prop:a} Suppose $d_1 = d_2 = \cdots d_{i-1} = 1$ for some $i\in \{1, \ldots, k\}$.  If the monomial
$$x_{F_1}^{d_1} \cdots x_{F_k}^{d_k} = x_{F_1} x_{F_2}\cdots x_{F_{i-1}} x_{F_i}^{d_i} \cdots x_{F_k}^{d_k}$$
is nonzero as an element of $A^d(M)$, then $\operatorname{rk}_{M}(F_j) = j$ for all $1 \leq j \leq i-1$.
\item[(b)]\label{prop:b} Furthermore, suppose $d_{i+1} = 1$.  If the monomial
$$x_{F_1}^{d_1} \cdots x_{F_k}^{d_k} = x_{F_1} x_{F_2}\cdots x_{F_{i-1}} x_{F_i}^{d_i} x_{F_{i+1}} x_{F_{i+2}}^{d_{i+2}} \cdots x_{F_k}^{d_k}$$
is nonzero as an element of $A^d(M)$, then $d_i = \operatorname{rk}_M(F_{i+1}) - \operatorname{rk}_M(F_{i-1}) -1$.
\end{enumerate}
\end{prop}

\begin{proof}
We prove the statement (a) first.  If $\operatorname{rk}_M(F_j) > j$ for some $j \in \{1, \ldots, i-1\}$, then there must exist $\ell\in \{1, \ldots, i-1\}$ such that $\operatorname{rk}_M(F_{\ell}) - \operatorname{rk}_M(F_{\ell-1}) \geq 2$.  Now, by construction every monomial appearing in
$$(\mathcal T_{F_k}^{d_k - 1} \circ \cdots\circ \mathcal T_{F_i}^{d_i - 1})(x_{F_1} x_{F_2}\cdots x_{F_{i-1}} x_{F_i}^{d_i} \cdots x_{F_k}^{d_k})$$
is square-free and does not contain a variable $x_G$ with $F_{\ell-1} \subsetneq G \subsetneq F_{\ell}$.  As each of these square-free monomials has degree $d$, that its support does not form a maximal chain in $\overline{\mathscr L_M}$ implies that it is zero.

We now use the statement (a) to prove the the statement (b).  If $d_i <  \operatorname{rk}_M(F_{i+1}) - \operatorname{rk}_M(F_{i-1}) -1$ then no monomials appearing in $\mathcal  T_{F_i}^{d_i - 1}(x_{F_1} x_{F_2}\cdots x_{F_{i-1}} x_{F_i}^{d_i} x_{F_{i+1}} x_{F_{i+2}}^{d_{i+2}} \cdots x_{F_k}^{d_k})$ can satisfy the condition given in statement (a).  If $d_i > \operatorname{rk}_M(F_{i+1}) - \operatorname{rk}_M(F_{i-1}) -1 $, then every monomial appearing in
$$\mathcal T_{F_i}^{\operatorname{rk}_M(F_{i+1}) - \operatorname{rk}_M(F_{i-1})-2}( x_{F_1} x_{F_2}\cdots x_{F_{i-1}} x_{F_i}^{d_i} x_{F_{i+1}} x_{F_{i+2}}^{d_{i+2}} \cdots x_{F_k}^{d_k})$$
is of the form $x_{F_1}\cdots x_{F_{i-1}}x_{G_1} \cdots x_{G_j} x_{F_i}^{d_i'}x_{G_{j+1}} \cdots x_{G_{r'}} x_{F_{i+1}} x_{F_{i+2}}^{d_{i+2}} \cdots x_{F_k}^{d_k}$ where $G_{j+1}$ covers $F_i$ and $F_i$ covers $G_j$.  As $d_i' = d_i - (\operatorname{rk}_M(F_{i+1}) - \operatorname{rk}_M(F_{i-1})-2) > 1$, we can apply $\mathcal T_{F_i}$ to these monomials again which makes them zero. 
\end{proof}

Combined with Proposition \ref{prop:topple}, the above Proposition \ref{prop:support} motivates the following modification of the toppling operation that is much more manageable for computation.

\begin{defn}
For $F\in \overline{\mathscr L_M}$, the \textbf{tight toppling operation} on degree $d$ homogeneous elements $\overline{\mathcal T_F}: \RR[x_F \ | \ F\in \overline{\mathscr L_M}]_d \to \RR[x_F \ | \ F \in \overline{\mathscr L_M}]_d$ \textbf{associated to $F$} is defined by its values on the monomials of degree $d$ as follows.  For a monomial $x_{F_1}^{d_1}\cdots x_{F_{k}}^{d_{k}}$, denote $r_j := \operatorname{rk}_M(F_j)$ for $j\in \{1, \ldots, k\}$.  We define

\smallskip
$\mathcal T_F(x_{F_1}^{d_1}\cdots x_{F_{k}}^{d_{k}}) :=$
$$
\begin{cases}
0  \quad & \textnormal{if $F_1, \ldots, F_k$ do not form a chain}\\[5mm]
\begin{split}
\textstyle\underset{\tiny\begin{array}{c} G\in (F_{i-1},F_i) \\ \operatorname{rk}_M(G) = r_{i-1} + 1\end{array}}\sum -\frac{|G\setminus F_{i-1}|}{|F_i\setminus F_{i-1}|}x_G\cdot x_{F_1}^{d_1} \cdots x_{F_i}^{d_i -1} \cdots x_{F_k}^{d_k} \\
\textstyle+ \underset{\tiny\begin{array}{c}G\in (F_i,F_{i+1}) \\ \operatorname{rk}_M(G) = d_i + r_i\end{array}}\sum  -\frac{|F_{i+1}\setminus G|}{|F_{i+1}\setminus F_i|} x_G\cdot x_{F_1}^{d_1} \cdots x_{F_i}^{d_i -1} \cdots x_{F_k}^{d_k}
\end{split} & \begin{array}{l}\textnormal{if $F_1\subsetneq \cdots \subsetneq F_k$ (relabeling if necessary)}\\
\textnormal{and $F_i = F$ with $d_i >1$} \\ \textnormal{and $d_j = 1$, $r_j = j$ for all $j\leq i-1$}\end{array}\\[11mm]
x_{F_1}^{d_1}\cdots x_{F_{k}}^{d_{k}} & \textnormal{otherwise}
\end{cases}.
$$
\end{defn}

\begin{lem}
For $f\in \RR[x_F \ | \ F\in \overline{\mathscr L_M}]$ homogeneous of degree $d$, we have $\overline{\mathcal T_F}(f) = f$ as elements in $A^d(M)$.
\end{lem}

\begin{proof}
The only nontrivial case is when $f$ has a monomial $x_{F_1}^{d_1} \cdots x_{F_k}^{d_k}$ satisfying $F_1 \subsetneq \cdots \subsetneq F_k$, $F_i = F$, $d_i > 1$, and $d_j = 1, r_j = j$ for all $j \leq i-1$.  In this case, apply $\mathcal T_{F_i}$ and consider the new variables $x_G$ that appear in the summations in \eqref{eq:topple}.  For $G\in (F_{i-1}, F_i)$ case we only need consider $G$ that cover $F_{i-1}$ by Proposition \ref{prop:support}(a), and for $G\in (F_i, F_{i+1})$ case we only need consider $G$ such that $\operatorname{rk}_M(G) - r_i - 1 = d_i -1$, i.e.\ $\operatorname{rk}_M(G) = d_i + r_i$, by Proposition \ref{prop:support}(b).
\end{proof}

We say that a monomial $x_{F_1}^{d_1} \cdots x_{F_k}^{d_k}$ is \textbf{initial (up to rank $i-1$)} if $\operatorname{rk}_M(F_j) = j$ for all $j\leq i-1$ where $i\in \{1, \ldots, k\}$ is the smallest such that $d_i > 1$ (and $i = k+1$ if $d_\ell = 1$ for all $\ell\in \{1, \ldots, k\}$).  Note that applying a tight toppling operation to an initial monomial of degree $d$ outputs a sum over initial monomials.  Thus, for a chain $\mathscr F: F_1 \subsetneq \cdots \subsetneq F_k$ in $\overline{\mathscr L_M}$ and $(d_1, \ldots, d_k) \in \ZZ^k_{>0}$ satisfying $d_1 + \cdots + d_k = d$, we can expand the monomial $x_{F_1}^{d_1} \cdots x_{F_k}^{d_k}$ into square-free monomials by
$$(\overline{\mathcal T_{F_k}}^{d_k - 1} \circ \cdots \circ \overline{\mathcal T_{F_1}}^{d_1 - 1})(x_{F_1}^{d_1} \cdots x_{F_k}^{d_k}).$$
We note an immediate consequence of this expansion process.

\begin{prop}
Let $\mathscr F: F_1 \subsetneq \cdots \subsetneq F_k$ be a chain in $\overline{\mathscr L_M}$, and $(d_1, \ldots, d_k) \in \ZZ^k_{>0}$ satisfying $d_1 + \cdots + d_k = d$.  Denote by $r_i := \operatorname{rk}_M(F_i)$ and $\widetilde d_i := \sum_{j=1}^i d_j$ (where $\widetilde d_0 := 0$).  We have
$$\deg_M(x_{F_1}^{d_1} \cdots x_{F_k}^{d_k}) \neq 0 \implies \widetilde d_{i-1} < r_i \leq \widetilde d_i \textnormal{ for all $i = 1, \ldots, k$}.$$
\end{prop}

\begin{proof}
For each $i\in \{1, \ldots, k\}$, the monomials with nonzero coefficients appearing in
$$(\overline{\mathcal T_{F_i}}^{d_i-1} \circ \cdots \circ \overline{\mathcal T_{F_1}}^{d_1 - 1})(x_{F_1}^{d_1} \cdots x_{F_k}^{d_k})$$
are initial up to rank at least $\widetilde d_j$.  By construction of the tight toppling operators, for these monomials to exist, it is necessary that $\widetilde d_{i-1} < r_i \leq \widetilde d_i$ for each $i \in\{ 1, \ldots, k \}$.
\end{proof}

We need one more proposition before we give a proof of Theorem \ref{main}.  It is the final ingredient that relates the coefficients of characteristic polynomials to the coefficients we obtain when expanding monomials by the tight toppling operations.  As we will be applying this to various matroid minors of the matroid $M$, the statements here are in terms of an arbitrary loopless matroid $M'$.

\begin{defn}
For $M'$ a loopless matroid on a ground set $E$ of rank $d+1$ and $-1\leq i \leq d+1$, define $\gamma(M',i)$ as follows.  For $i=-1$ and $i=d+1$, define $\gamma(M',-1) = -1$ and $\gamma(M', d+1) = 0$; for $0\leq i \leq d$, define
$$ \gamma(M', i) := \sum_{\mathscr G \in \mathscr L_{M'}^{\leq i}} (-1) \left(-\frac{|G_1\setminus G_0|}{|G_1|} \right)\left(-\frac{|G_2\setminus G_1|}{|G_2|} \right)\cdots \left(-\frac{|G_{i+1}\setminus G_i|}{|G_{i+1}|} \right)$$
where $\mathscr L_{M'}^{\leq i}$ consists of chains $\mathscr G: \emptyset = G_0 \subsetneq G_1 \subsetneq \cdots \subsetneq G_i\subsetneq G_{i+1} = E$ such that $\operatorname{rk} G_j= j$ for $j=0,\ldots, i$.
\end{defn}

\begin{prop} \label{prop:gamma} Let $M'$ be a loopless matroid of rank $r = d+1$ on a ground set $E$, and let $\overline \chi_{M'}(t) = \chi_{M'}(t)/(t-1)$ be the reduced characteristic polynomial of $M'$.
\begin{enumerate}
\item[(a)] We have $\displaystyle \sum_{\mathscr G\in \mathscr L_{M'}^{\leq d}} \frac{|G_1|}{|E|} \cdot \frac{|G_2\setminus G_1|}{|E\setminus G_1|} \cdots \frac{|G_d \setminus G_{d-1}|}{|E\setminus G_{d-1}|} = 1$, and\\[3mm]
\item[(b)] $\overline \chi_{M'}(t) = \sum_{i=0}^d \gamma(M',i)t^{d-i}$.
\end{enumerate}
\end{prop}

\begin{proof}
We prove the statement (a) first by induction on $d$.  For the base case $d=1$, as $M'$ is loopless, the statement follows from the cover-partition axiom (F3) in Definition \ref{defn:flats} applied to $\emptyset \in \mathscr L_{M'}$.  For the induction step, we note that
$$\sum_{\mathscr G\in \mathscr L_{M'}^{\leq d}} \frac{|G_1|}{|E|} \cdot \frac{|G_2\setminus G_1|}{|E\setminus G_1|} \cdots \frac{|G_d \setminus G_{d-1}|}{|E\setminus G_{d-1}|} = \sum_{G_1 \in \mathfrak A(M')} \frac{|G_1|}{|E|} \sum_{\mathscr G' \in \mathscr L^{\leq d-1}_{M'/G_1}} \frac{|G_2\setminus G_1|}{|E\setminus G_1|} \cdots \frac{|G_d \setminus G_{d-1}|}{|E\setminus G_{d-1}|}$$
where $\mathscr G'$ ranges over the maximal chain of nonempty proper flats in the contraction $M'/G_1$, which has rank one smaller than that of $M'$.  Thus, by induction hypothesis, we get
$$\sum_{G_1 \in \mathfrak A(M')} \frac{|G_1|}{|E|} \sum_{\mathscr G' \in \mathscr L^{\leq d-1}_{M'/G_1}} \frac{|G_2\setminus G_1|}{|E\setminus G_1|} \cdots \frac{|G_d \setminus G_{d-1}|}{|E\setminus G_{d-1}|} = \sum_{G_1 \in \mathfrak A(M')} \frac{|G_1|}{|E|} = 1$$
as desired.

For the proof of statement (b), we begin by recalling a useful theorem.  Denote by $\mu(\cdot, \cdot)$ the M\"obius function on the lattice $\mathscr L_{M'}$.  Weisner's theorem \cite[\S3.9]{Sta12} states that for $G\in \mathscr L_{M'}$ a flat and $a\in G$ we have
$$\mu(\emptyset, G) = -\sum_{a\notin F \lessdot G} \mu(\emptyset, F)$$
where $F\lessdot G$ means $G$ covers $F$ in the lattice $\mathscr L_{M'}$.  Thus, we have that
$$|G|\mu(\emptyset, G) = -\sum_{a\in G}\sum_{a\notin F \lessdot G} \mu(\emptyset, F) = -\sum_{F\lessdot G} |G\setminus F| \mu(\emptyset, F).$$
In other words, we have $ \mu(\emptyset, G) = \sum_{F\lessdot G} -\frac{|G\setminus F|}{|G|} \mu(\emptyset, F)$.  Repeatedly applying this identity gives
$$\mu(\emptyset, G) =  -\gamma(M'|G, \operatorname{rk} G -1).$$
Now, since one can define $\gamma(M',i)$ recursively as
$$\gamma(M',i) := \begin{cases}
-1 & \textnormal{for $i = -1$}\\[2mm]
\displaystyle \sum_{\operatorname{rk} F = i} - \frac{|E\setminus F|}{|E|} \gamma(M'|F, i -1) & \textnormal{for $i\geq 0$},
\end{cases}$$
we have
$$\gamma(M',i) = \sum_{\operatorname{rk} F = i}  \frac{|E\setminus F|}{|E|} \mu(\emptyset, F)$$
for $0\leq i \leq d+1$.  Now, recall that the characteristic polynomial of $M'$ is 
$$\chi_{M'}(t)  = \sum_{G\in \mathscr L_{M'}} \mu(\emptyset, G) t^{r - \operatorname{rk}_{M'}(G)},$$
and since $\chi_{M'}(t) = (t-1)\overline{\chi}_{M'}(t)$, to show the desired equality in the statement (b), we need show that $\gamma(M', 0) = 1$ and $\gamma(M',i+1) - \gamma(M', i) = \sum_{\operatorname{rk} G = i+1} \mu(\emptyset, G)$ for $0\leq i \leq d$.  Well, $\gamma(M',0) = 1$ is immediate, and 
\begin{align*}
\gamma(M',i+1)- \sum_{\operatorname{rk} G = i+1}\mu(\emptyset, G) &= \sum_{\operatorname{rk} G = i+1} \left(\frac{|E\setminus G|}{|E|}-1\right)\mu(\emptyset, G) \\
&=   \sum_{\operatorname{rk} G = i+1} -\frac{|G|}{|E|}\mu(\emptyset, G)\\
&=  \sum_{\operatorname{rk} G = i+1}\sum_{F\lessdot G} \frac{|G\setminus F|}{|E|} \mu(\emptyset, F)\\
&=  \sum_{\operatorname{rk} F = i} \frac{|E\setminus F|}{|E|} \mu(\emptyset, F) = \gamma(M',i)
\end{align*}
where the last equality follows from the cover-partition axiom (Definition \ref{defn:flats} (F3)).
\end{proof}

\medskip
Following \cite{AHK18} and denoting by $\mu^i(M')$ the (unsigned) coefficient of $t^{\operatorname{rk} M' -1 - i}$ in the reduced characteristic polynomial of $M'$, Proposition \ref{prop:gamma}(b) says that $|\gamma(M',i)| = \mu^i(M')$.  We are now finally ready to prove the main theorem of this section.

\begin{proof}[Proof of Theorem \ref{main}]
Note that ${d_i - 1 \choose \widetilde d_i -r_i}$ is nonzero if and only if $\widetilde d_{i-1} < r_i \leq \widetilde d_i$ since $(r_i - \widetilde d_{i-1} -1) + (\widetilde d_i - r_i) = d_i - 1$, so the condition for the expression in Theorem \ref{main} to be nonzero agrees with Proposition \ref{prop:support}.  Each monomial appearing in
$$(\overline{\mathcal T_{F_i}}^{d_i-1} \circ \cdots \circ \overline{\mathcal T_{F_1}}^{d_1 - 1})(x_{F_1}^{d_1} \cdots x_{F_k}^{d_k})$$
is of the form
$$x_{F_1^{(1)}}x_{F_1^{(2)}} \cdots x_{F_1^{(\widetilde d_1)}}x_{F_2^{(\widetilde d_1 +1)}}x_{F_1^{(\widetilde d_1 + 2)}} \cdots x_{F_2^{(\widetilde d_2)}}\cdots x_{F_k^{(\widetilde d_{k-1} +1)}}x_{F_k^{(2)}} \cdots x_{F_k^{(\widetilde d_k)}}$$
where $\operatorname{rk}_M(F_i^{(\ell)}) = \ell$ and $F_i^{(r_i)} = F_i$ for all $i = 1, \ldots, k$, and has coefficient
\begin{equation}\label{eq:coeff}
(-1)^{\sum_i (d_i -1)} \prod_{i=1}^k \left( {d_i -1 \choose \widetilde d_i - r_i} \cdot
\prod_{j=1}^{r_i - \widetilde d_{i-1}-1} \frac{|F_i^{(\widetilde d_{i-1} +j)}\setminus F_{i-1}^{(\widetilde d_{i-1}+j-1)}|}{|F_i \setminus  F_{i-1}^{(\widetilde d_{i-1}+j-1)}|} \cdot
\prod_{j=1}^{\widetilde d_i - r_i}\frac{|F_i^{(r_i + j+1)}\setminus F_i^{(r_i +j)}|}{|F_i^{(r_i +j + 1)}\setminus F_i|} \right).
\end{equation}
As all these monomials (ignoring the coefficient part) evaluate to 1 under $\deg_M$, we are summing the coefficients \eqref{eq:coeff} over all possible maximal chains $F_1^{(1)} \subsetneq \cdots \subsetneq F_k^{(\widetilde d_k)}$ with $F_i^{(r_i)} = F_i$.  First consider collecting the coefficients \eqref{eq:coeff} together by ones that have common $\prod_{j=1}^{\widetilde d_i - r_i}\frac{|F_i^{(r_i + j+1)}\setminus F_i^{(r_i +j)}|}{|F_i^{(r_i +j + 1)}\setminus F_i|}$ part.  By applying Proposition \ref{prop:gamma}.(a) for $M' = M|F_i/F_{i-1}^{(\widetilde d_{i-1})}$, we have that in each collection the $\prod_{j=1}^{r_i - \widetilde d_{i-1}-1} \frac{|F_i^{(\widetilde d_{i-1} +j)}\setminus F_{i-1}^{(\widetilde d_{i-1}+j-1)}|}{|F_i \setminus  F_{i-1}^{(\widetilde d_{i-1}+j-1)}|}$ part collect to give sum equal to 1.  Thus, we are now summing
\begin{equation}\label{eq:coeff2}
(-1)^{\sum_i (d_i -1)} \prod_{i=1}^k \left( {d_i -1 \choose \widetilde d_i - r_i} \cdot
\prod_{j=1}^{\widetilde d_i - r_i}\frac{|F_i^{(r_i + j+1)}\setminus F_i^{(r_i +j)}|}{|F_i^{(r_i +j + 1)}\setminus F_i|} \right).
\end{equation}
over all chains $F_1^{(r_1)} \subsetneq \cdots \subsetneq F_1^{(\widetilde d_1)} \subsetneq \cdots \subsetneq F_k^{(r_k)} \subsetneq \cdots \subsetneq F_k^{(\widetilde d_k)}$ where $F_i^{(r_i)} = F_i$ and $\operatorname{rk}_M(F_i^{(\ell)}) = \ell$ for all $i = 1, \ldots, k$.  In other words, our sum is
$$(-1)^{\sum_i (d_i -1)} \prod_{i=1}^k \left( {d_i -1 \choose \widetilde d_i - r_i} \cdot |\gamma(M|F_{i+1}/F_i, \widetilde d_i - r_i)| \right),$$
and applying Proposition \ref{prop:gamma} (b) gives our desired result.
\end{proof}

\medskip
\section{First applications of the volume polynomial}

We give some first applications of the volume polynomial of a matroid.  In this section, we always set the ground set of a matroid to be $[n] := \{1, \ldots, n\}$ for some $n$.

\subsection*{Volumes of generalized permutohedra}

As an immediate application, we compute the volume of a generalized permutahedron using the volume polynomial of a boolean matroid.  Postnikov gave a formula \cite[Corollary 9.4]{Pos09} for the generalizated permutohedra that arise as Minkowski sums of simplices.  Here we give a new explicit formula for volumes of any generalized permutohedra.

\medskip
A \textbf{generalized permutohedron} $P$ is obtained by sliding the facets of the permutohedron.  More precisely, for a submodular function $z_{(\cdot)}: 2^{[n]} \to \RR$ (where $[n]\supset I \mapsto z_I$) on the boolean lattice $\mathscr L_{U_{n,n}}$, we define
$$P^{\leq}(\underline z) := \{ (\underline x) \in \RR^n \ | \ \sum_{i\in [n]} x_i = z_{[n]}, \ \sum_{i\in I} x_i \leq z_I \ \forall I\subset [n]\}$$
which is a polytope of dimension at most $n-1$ in $\RR^n$.

For nonnegative set of numbers $\{y_I \ | \ I \subset [n]\}$, one can consider a Minkowski sum
$$P^+(\underline y) := \sum_{I\subset [n]} y_I \Delta_I$$
where $\Delta_I = \operatorname{Conv}(e_i : i \in I) \subset \RR^n$.  $P^+(\underline y)$ is also a generalied permutohedra, and by setting $z_I = \sum_{J\subset I} y_J$, one has $P^\leq(\underline z) = P^+(\underline y)$ (\cite[Proposition 6.3]{Pos09}).  Note that not every generalized permutohedra is realized as $P^+(\underline y)$ however (\cite[Remark 6.4]{Pos09}).  The volume of $P^+(\underline y)$ was computed by Postnikov in \cite{Pos09}.

\begin{thm}\cite[Corollary 9.4]{Pos09}
The volume of $P^+(\underline y)$ is
$$ \operatorname{Vol} P^+(\underline y) = \frac{1}{(n-1)!} \sum_{(I_1, \ldots, I_{n-1})} y_{I_1} \cdots y_{I_{n-1}}$$
where the summation is over ordered subsets $(I_1, \ldots, I_{n-1})$ of $[n]$ such that for any  $\{i_1, \ldots, i_k\}\subset [n-1]$ we have $| I_{i_1} \cup \cdots \cup I_{i_k}| \geq k+1$.
\end{thm}

To convert the above formula into one in terms of $z_I$'s, one uses M\"obius inversion formula to get
$$z_I = \sum_{I\subset J} y_J \iff y_I  = \sum_{J\subset I}z_I \mu(J,I) = \sum_{J\subset I} (-1)^{|I\setminus J|} z_J.$$

Here we give another formula in terms of the $z_I$'s.  The Bergman fan $\Sigma_{A_{n-1}}$ of $U_{n,n}$ is the normal fan of the permutohedron, whose rays correspond to subsets of $[n]$.  As nef torus invariant divisors on the toric variety $X_{\Sigma_n}$ exactly correspond to submodular functions with $z_\emptyset = z_{[n]} = 0$, well-known results from toric geometry on volumes of nef torus invariant divisors (\cite[Theorem 13.4.3]{CLS11}) imply that
$$\operatorname{Vol} P(\underline z) = \frac{1}{(n-1)!} VP_{U_{n,n}}(\underline z).$$
Note that the flats of $U_{n,n}$ are just subsets of $[n]$.  As $U_{n,n}|I/J \simeq U_{m,m}$ where $m = | I \setminus J|$, Corollary \ref{volpol} gives us the following formula.

\begin{prop}\label{volGP}
Let $z_{(\cdot )}: I \mapsto z_I\in \RR$ be a submodular function such that $z_\emptyset = z_{[n]} =0$, then the volume of the generalized permutohedron $P(\underline z)$ is
$$(n-1)! \operatorname{Vol} P(\underline z) = \sum_{I_\bullet, \underline d} (-1)^{d-k} {d \choose d_1, \ldots, d_k}\prod_{i=1}^k {d_i -1 \choose \widetilde d_i - |I_i|} {|I_{i+1}| - |I_i| - 1 \choose \widetilde d_i - |I_i|}z_{I_i}$$
where the summation is over chains $\emptyset \subsetneq I_1 \subsetneq \cdots \subsetneq I_k\subsetneq I_{k+1} = [n]$ and $\underline d = (d_1, \ldots, d_k)$ such that $\sum d_i = n-1$ and $\widetilde d_j:= \sum_{i=1}^j d_i$.
\end{prop}

\begin{rem}
Setting $z_I = \frac{(n - |I| ) |I|}{2}$, one recovers that the volume of the permutohedron $P_n := \operatorname{Conv}( (\sigma(n), \sigma(n-1), \ldots, \sigma(1))\in \RR^n \ | \ \sigma\in S_n)$ is $n^{n-2}$.
\end{rem}

It was not clear to the author whether one could derive the formula in the above Proposition \label{GPvol} directly from \cite[Corollary 9.4]{Pos09}.

\begin{ques}\label{GPvolques}
Can one derive the formula in Proposition \ref{volGP} directly from \cite[Corollary 9.4]{Pos09}, or vice versa?
\end{ques}

\medskip
\subsection*{Valuativeness of the volume polynomial}  While the volume polynomial $VP_M$ is an alternate but equivalent algebraic encoding of the Chow ring $A^\bullet(M)$, it lends itself more naturally as a function on a matroid when viewed as a map $M\mapsto VP_M \in \RR[t_S : S\in 2^{[n]}]$.   In this subsection, we illustrate this by showing that $M \mapsto VP_M$ is a valuation under matroid polytope subdivisions, a statement that would not make sense for $M \mapsto A^\bullet(M)$.

\medskip
We first give a brief sketch on matroid polytopes and matroid valuations; for more on matroid valuations, we point to \cite{AFR10} and \cite{DF10}.  Given a matroid $M$ on $[n]$ of rank $r$ with bases $\mathcal B$, its \textbf{matroid polytope} is defined as 
$$\Delta(M) := \operatorname{Conv}(e_B \ | \ B\in \mathcal B)\subset \RR^n$$
where $e_S := \sum_{i\in S} e_i$ for $S\subset [n]$ and $e_i$'s are the standard basis of $\RR^n$.  Its vertices are the indicator vectors for the bases of $M$, and it follows from a theorem of Gelfand, Goresky, MacPherson, and Serganova that the faces of $\Delta(M)$ are also matroid polytopes (\cite{GGMS87}).   Given a matroid polytope $Q$, we denote by $M_Q$ the corresponding matroid.  A \textbf{matroid subdivision} $\mathcal S$ of a matroid polytope $\Delta$ is a polyhedral subdivision $\mathcal S: \Delta = \bigcup_i \Delta(M_i)$ such that each $\Delta(M_i)$ is a matroid polytope of some matroid $M_i$ (necessarily of rank $r$ on with ground set $[n]$). Denote by $\operatorname{Int}(\mathcal S)$ the faces of $\Delta(M_i)$'s that is not on the boundary of $\Delta(M)$.  It is often of interest to see whether a function on matroids behave well via inclusion-exclusion with respect to matroid subdivisions:

\begin{defn}
Let $R$ be an abelian group, and let $\mathcal M := \bigcup_{n\geq 0} \{\textnormal{matroids on ground set $[n]$}\}$ be the set of all matroids.  A map $\varphi: \mathcal M \to R$ is a \textbf{matroid valuation} (or is \textbf{valuative}) if for any $M\in \mathcal M$ and a matroid subdivision $\mathcal S: \Delta(M) = \bigsqcup_i \Delta(M_i)$ one has
$$\varphi(M) = \sum_{Q \in \operatorname{Int}(\mathcal S)} (-1)^{\dim \Delta(M) - \dim Q}\varphi(Q_M).$$
\end{defn}

Many interesting functions on matroids are matroid valuations, for example the Tutte polynomial (\cite[Corollary 5.7]{AFR10}) and the quasi-symmetric functions $\chi^{BJR}$ introduced by Billera, Jia, and Reiner in \cite{BJR09}.  It was also shown recently in \cite{LdMRS17} that matroid analogues of Chern-Schwartz-MacPherson (CSM) classes are matroid valuations.  Here we show that the volume polynomial of a matroid is also a matroid valuation.

\begin{prop}\label{val}
$M\mapsto VP_M(t) \in \RR[t_S : S\in 2^{[n]}]$ is a matroid valuation.
\end{prop}

We state a useful lemma before the proof.

\begin{lem}
Let $f,g: \mathcal M \to \QQ$ be matroid valuations, and let $A\subset [n]$.  Then $f*_A g$ defined by $(f*_A g)(M) = f(M|_A)g(M/A)$ is also a matroid valuation.
\end{lem}

\begin{proof}
We follow the notation as set in \cite[\S5]{DF10}.  By \cite[Corollary 5.5, 5.6]{DF10}, we reduce to the case when $f = \sum_{i=0}^\infty s_{\underline X^{(i)}, \underline q^{(i)}}$ and $g = \sum_{i=0}^\infty s_{\underline Y^{(i)}, \underline r^{(i)}}$, where $\underline X^{(i)}$ is a flag $\emptyset =X_0^{(i)}\subsetneq X_1^{(i)} \subsetneq \cdots \subsetneq X_{k_i}^{(i)} = [i]$ and $\underline q^{(i)}$ is an increasing sequence $0 = q_0^{(i)} \leq q_1^{(i)} < q_2^{(i)} < \cdots < q_{k_i}^{(i)}$ (likewise for $\underline Y$ and $\underline r$), and $s_{\underline X^{(i)}, \underline q^{(i)}}$ is a function on the set of matroids with ground set $[i]$ defined by
$$
s_{\underline X^{(i)}, \underline q^{(i)}}(M) = 
\begin{cases}
1 & \textnormal{if $M$ has ground set $[i]$ and $\operatorname{rk}_M(X_j^{(i)}) = q_j^{(i)}$ for all $j = 0, \ldots, i$}\\
0 & \textnormal{otherwise}.
\end{cases}
$$
Let $|A| = \ell$.  Then, $f*_A g$ (where the ground set $A\subset [n]$ of $M|_A$ is relabelled to be $[\ell]$ by increasing order and likewise for $[n]\setminus A$ for $M/A$) is equal to $s_{\underline Z, \underline p}$ where $\underline Z$ is the concatenation of $\underline X^{(\ell)}$ and $\underline Y^{(n-\ell)} \cup A$, and $\underline p$ is the concatenation of $\underline q^{(\ell)}$ and $\underline r^{(n-\ell)} + q^{(\ell)}_{k_\ell}$ (by union or plus we mean adding to each element in the sequence).  Hence, $f *_A g$ is also a matroid valuation.
\end{proof}

\begin{proof}[Proof of Proposition \ref{val}]
We show that the maps $M\mapsto (\textnormal{coefficient of } t_{S_1}^{d_1}\cdots t_{S_k}^{d_k} \textnormal{ in } VP_M)$ are matroid valuations.  As taking the Tutte polynomial is valuative, the map $M\mapsto \overline\chi_M(t)$ is valuative, and as a result, the map $M\mapsto \mu^i(M)$ is valuative for any $i \geq 0$.  Moreover, note that if $\varphi: \mathcal M \to R$ is a matroid valuation, then so is $\widetilde\varphi(n,r)$ for any fixed $n$ and $r$, where $\widetilde\varphi(n,r)$ is defined by
$$
\widetilde\varphi(n,r)(M) := 
\begin{cases}
\varphi(M) & \textnormal{if $M$ is a matroid of rank $r$ on $n$ elements}\\
0 & \textnormal{otherwise}.
\end{cases}
$$
Thus, for $s,t,u,v\in \ZZ_{\geq 0}$ such that $s\geq t$ and $u\geq v$, we define
$$
\widetilde\mu(s,t,u,v)= 
\begin{cases}
0 & \textnormal{if $\operatorname{rk}(M) \neq u-v$}\\
\frac{1}{t!} {t-1 \choose s-v} \mu^{s-v}(M) & \textnormal{if $\operatorname{rk}(M) = u-v$}.
\end{cases}
$$
Then for a flag $\emptyset \subsetneq S_1 \subsetneq \cdots \subsetneq S_k \subsetneq S_{k+1} = [n]$, an increasing sequence $0 = \widetilde d_0 < \widetilde d_1 < \cdots < \widetilde d_k = d$, and an increasing sequence $0 = r_0 < r_1 < \cdots < r_k < r_{k+1} = d+1$, we have that (letting $d_i := \widetilde d_i - \widetilde d_{i-1}$)
$$(-1)^{d-k}d! \widetilde\mu (\widetilde d_0, d_0, r_1, r_0) *_{S_1} \widetilde \mu(\widetilde d_1, d_1, r_2, r_1) *_{S_2} \cdots *_{S_k} \widetilde\mu(\widetilde d_k, d_k, r_{k+1}, r_k)$$
is a matroid valuation that evaluates a matroid $M$ to $(-1)^{d-k}{d \choose d_1, \ldots, d_k} \prod_{i=1}^k {d_i - 1\choose \widetilde d_i - r_i} \mu^{\widetilde d_i - r_i}(M|S_{i+1}/S_i)$ if $M$ is of rank $d+1$ with $S_\bullet$ being a chain of flats with $\operatorname{rk}_M(S_i) = r_i$, and 0 otherwise.  Thus, summing over all sequences $r_\bullet$ of the above function, we have that taking the coefficient of $t_{S_1}^{d_1}\cdots t_{S_k}^{d_k}$ in the $VP_M$ is a matroid valuation.
\end{proof}

That the matroid volume polynomial behaves well with respect to matroid polytope subdivisions and the appearance of matroid-minor Hopf monoid structure appearing in its expression as in the proof above suggests that there may be a generalization of Chow ring of matroids to Coxeter matroids of arbitrary Lie type (where matroids are the type $A$ case).  For Coxeter matroids see \cite{BGW03}.

\begin{ques}\label{ADE}
Is there are Hodge theory of Coxeter matroids, generalizing the Hodge theory of matroids as described in \cite{AHK18}?
\end{ques}

\section{The shifted rank volume of a matroid}

Let $M = (E,\mathscr L_M)$ be a matroid.  Following \cite{AHK18}, a (strictly) submodular function $c_{(\cdot)}: 2^E \to \RR$ with $c_\emptyset = c_E = 0$ gives a combinatorially \emph{nef (ample)} divisor $D = \sum_{F\in \overline{\mathscr L_M}} c_Fx_F \in A^1(M)$.  For realizable matroids, if the divisor $D\in A^1(M)$ is combinatorially nef (ample) then as an element of $A^1(X_{\mathscr R(M)})$ the divisor $D$ is nef (ample) in the classical sense.  As the rank function $\operatorname{rk}_M(\cdot)$ is a distinguished submodular function\footnote{Technically, we have $\operatorname{rk}_M(E) > 0$, but a function $c_{(\cdot)}: 2^E \to \RR$ defined by $c_I := \operatorname{rk}_M(I)$ for all $ I\subsetneq E$ and $c_E := 0$ is still submodular.} of a matroid, we define the following notions.

\begin{defn}\label{voldefn}
For a matroid $M$, define its \textbf{shifted rank divisor} $D_M$ to be
$$D_M := \sum_{F\in \overline{\mathscr L_M}} (\operatorname{rk} F) x_F,$$
and define the \textbf{shifted rank volume of a matroid} $M$ to be the volume of its shifted rank divisor:
$$\operatorname{shRVol}(M) := \deg\Big(\sum_{F\in \overline{\mathscr L_M}} (\operatorname{rk} F)x_F\Big)^{\operatorname{rk} M-1}.$$
\end{defn}

A slight modification of the proof of Proposition \ref{val} implies that the shifted rank volume of a matroid is a valuative invariant.

\begin{cor}
The map $M\mapsto \operatorname{shRVol}(M)$ is a matroid valuation.
\end{cor}

We remark that in a forthcoming work \cite{Eur19}, the author discusses various geometrically motivated divisors including what the author calls the rank divisor, to which the shifted rank divisor here is closely related.  Moreover, the author gives a formula for the shifted rank volume in terms of beta invariants, but the author is not aware of any nice combinatorial meaning of the shifted rank volume of a matroid.  In fact, it seems to be a genuinely new invariant of a matroid.

\begin{rem}[Relation to other invariants, or lack thereof] \label{volrel}  We point to \verb+volumeMatroid.m2+ for computations supporting the statements below.

\begin{itemize}
\item Volume $\operatorname{shRVol}(M)$ is not a complete invariant.
\item Same Tutte polynomial does not imply same volume, and vice versa.  The two graphs in Figure 2 of \cite{CLP15} have the same Tutte polynomial but their matroids are not isomorphic; their shifted rank volumes are 1533457 and 1534702.  There are many examples of matroids with same volume but with different Tutte polynomials.
\item Same volume of the matroid polytope does not imply same shifted rank volume, and vice versa.
\end{itemize}

\end{rem}

\begin{rem}
Since the two graphs in the Figure 2 of \cite{CLP15} have the same Tutte polynomial but different matroid volumes, we thus obtain an example of two \emph{simple} matroids with same Tutte polynomials but with different $\mathcal G$-invariant (see \cite{Der09} for $\mathcal G$-invariant of a matroid).
\end{rem}

For realizable matroids however, the volume measures how general the associated hyperplane arrngement is.

\begin{thm}\label{realmax}
Let $M$ be a realizable matroid of rank $r$ on $n$ elements.  Then 
$$\operatorname{shRVol}(M) \leq \operatorname{shRVol}(U_{r,n}) = n^{r-1} \quad \textnormal{ with equality iff $M = U_{r,n}$.}
$$
\end{thm}

The proof is algebro-geometric in nature, and follows from the following.

\begin{prop}\label{subsys}
Let $M$ be a simple realizable matroid of rank $r = d+1$ on a ground set $E = \{1, \ldots, n\}$.  Let $\mathscr R(M)$ be a realization of $M$ over an algebraically closed field $\mathbbm k$, and let $\pi: X_{\mathscr R(M)} \to \PP V^* \simeq \PP^d_{\mathbbm k}$ be the blow-down map  as described in Definition \ref{defn:wndcpt}.  Then the complete linear series $|D_M|$ of the divisor $D_M$ is a linear subseries of $|\mathscr \pi^*(\mathscr O_{\PP_{\mathbbm k}^d}(n))|$, with equality iff $M = U_{r,n}$.

More precisely, $|D_M|$ is isomorphic to the linear subseries $L \subset |\mathscr O_{\PP^d}(n)|$ consisting of homogeneous degree $n$ polynomials on $\PP^d$ vanishing on $L_F$ with order at least $|F| -\operatorname{rk} F$.
\end{prop}

\begin{proof}
Let $\mathcal A_{\mathscr R(M)} = \{L_i\}_{i\in E} \subset \PP V^* \simeq \PP^d$ be the hyperplane arrangement associated to the realization $\mathscr R(M)$ of $M$,  and denote $h := c_1(\mathscr O_{\PP^d}(1))$ the hyperplane class.  Then for any $i \in E$, we have
$$\pi^* h = \pi^*[L_i] =  \sum_{i\in F} x_F$$
by the construction of $X_{\mathscr R(M)}$ as consecutive blow-ups.  Thus, we have
\begin{align*}
 D_M = \sum_F (\operatorname{rk} F) x_F =&  \sum_{i\in E} x_i + \sum_{\operatorname{rk} F >1} (\operatorname{rk} F)x_F \\
=&  \sum_{i\in  E}\Big(\pi^*h - \sum_{i \subsetneq F} x_F\Big) + \sum_{\operatorname{rk} F > 1}(\operatorname{rk} F)x_F \\
=&  n\pi^* h + \sum_{\operatorname{rk} F > 1}(\operatorname{rk} F - |F|)x_F\\
=&  n\pi^* h - \sum_{\operatorname{rk} F > 1}(|F| - \operatorname{rk} F)x_F.
\end{align*}
Hence, noting that the rank function on $M$ satisfies $|S| \geq \operatorname{rk} S$ for any subset $S\subset E$, we see that the divisors in $|D_M|$ are exactly the divisors in $|\pi^*\mathscr O_{\PP_k^d}(n)|$ that vanish on $\widetilde L_F$ with order at least $|F| -\operatorname{rk} F \geq 0$.  Moreover, since $	|F| -\operatorname{rk} F \geq |G| - \operatorname{rk} G$ for any flats $F\supset G$, the divisors in $|D_M|$ are in fact elements of $|\mathscr O_{\PP^d}(n)|$ vanishing on $L_F$ with order at least $|F| = \operatorname{rk} F$.  Lastly, the only simple matroids with the property $|F| = \operatorname{rk} F$ for all flats $F\in \overline{\mathscr  L_M}$ are the uniform matroids.  (Given an $r$-subset $S$, its closure $\overline S$ is a flat not equal to $E$ iff $S$ is not a basis).
\end{proof}

Two immediate consequences follow.

\begin{cor}
The volume of $U_{r,n}$ is $n^{r-1}$.
\end{cor}

\begin{proof}
The map given by $\pi^*\mathscr O_{\PP^d}(n)$ factors as $X_{\mathscr R(M)} \overset{\pi}{\to} \PP^d \overset{\nu_n} \to \PP^{{d+n \choose d}-1}$ where $\nu_n$ is the $n$-tuple Veronese embedding, and the degree of the Veronese embedding is $n^d = n^{r-1}$.
\end{proof}

\begin{proof}[Proof of Theorem \ref{realmax}]
First assume $M$ is simple.  Then Proposition \ref{subsys} shows that we have an inclusion of section rings $R(D_M)_\bullet \subset R(\pi^*\mathscr O_{\PP^r}(n))_\bullet$ with equality iff $M = U_{r,n}$.  As both divisors $D_M$ and $\pi^* \mathscr O_{\PP^r}(n)$ are nef divisors on $X_{\mathscr R(M)}$, the volume polynomial agrees with the volume of a divisor in the classical sense.  If $M$ is not simple, then its volume is the same as the simple matroid $M'$ satifying $\mathscr L_M = \mathscr L_{M'}$, and $M'$ has at most $n-1$ elements.
\end{proof}

In a forthcoming paper \cite{Eur19}, the author gives a proof for Theorem \ref{realmax} without the realizability condition.  However, the proof is not a combinatorial reflection of the geometric proof given here.  Trying to apply similar method as in the proof of Theorem \ref{realmax} for the non-realizable matroids naturally leads to the following question.

\begin{ques}\label{NObody}
Is there a naturally associated convex body of which this volume polynomial is measuring the volume?  In other words, is there a theory of Newton-Okounkov bodies for general matroids (a.k.a.\ linear tropical varieties)?
\end{ques}

On the flip side of maximal volumes, we have the following conjecture on the minimal values.  We have confirmed the validity of the conjecture up to all matroids on 8 elements.

\begin{conj}\label{min}
The minimum volume among simple matroids of rank $r$ on $n$ is achieved uniquely by the matroid $U_{r-2,r-2}\oplus U_{2,n-r+2}$, and its volume is $r^{r-2}((n-r+1)(r-1)+1)$.
\end{conj}

\section{Examples}

All the matroids in the examples are realizable to illustrate the algebro-geometric connections.  For each case let $\mathscr R(M)$ be a realization of a realizable matroid $M$.  We start with two examples of rank 3, whose associated wonderful compactifications are obtained from blowing-up points on $\PP^2$.  For these surfaces, we can compute the intersection numbers with classical algebraic geometry without much difficulty; see \cite[\S V.3]{Har77}.  We check in these examples that the classical results and the combinatorial ones introduced in this paper indeed agree.

\begin{eg} Let $M := U_{3,4}$.  Its lattice of flats $\mathscr L_M$ and the hyperplane arrangement $\mathcal A_{\mathscr R(M)}\subset \PP^2$ are given as follows.
\begin{center}
 \begin{tikzpicture}[scale=.8, vertices/.style={draw, fill=black, circle, inner sep=0pt}]
                \node [vertices, label=right:{${}$}] (0) at (-0+0,0){};
                \node [vertices, label=right:{${3}$}] (1) at (-2.25+0,1.33333){};
                \node [vertices, label=right:{${2}$}] (2) at (-2.25+1.5,1.33333){};
                \node [vertices, label=right:{${1}$}] (3) at (-2.25+3,1.33333){};
                \node [vertices, label=right:{${0}$}] (4) at (-2.25+4.5,1.33333){};
                \node [vertices, label=right:{${2, 3}$}] (5) at (-3.75+0,2.66667){};
                \node [vertices, label=right:{${1, 2}$}] (8) at (-3.75+1.5,2.66667){};
                \node [vertices, label=right:{${1, 3}$}] (6) at (-3.75+3,2.66667){};
                \node [vertices, label=right:{${0, 2}$}] (9) at (-3.75+4.5,2.66667){};
                \node [vertices, label=right:{${0, 1}$}] (10) at (-3.75+6,2.66667){};
                \node [vertices, label=right:{${0, 3}$}] (7) at (-3.75+7.5,2.66667){};
                \node [vertices, label=right:{${0, 1, 2, 3}$}] (11) at (-0+0,4){};
        \foreach \to/\from in {0/1, 0/2, 0/3, 0/4, 1/5, 1/6, 1/7, 2/8, 2/5, 2/9, 3/8, 3/6, 3/10, 4/9, 4/10, 4/7, 5/11, 6/11, 7/11, 8/11, 9/11, 10/11}
        \draw [-] (\to)--(\from);
\end{tikzpicture}
\includegraphics[height=43mm]{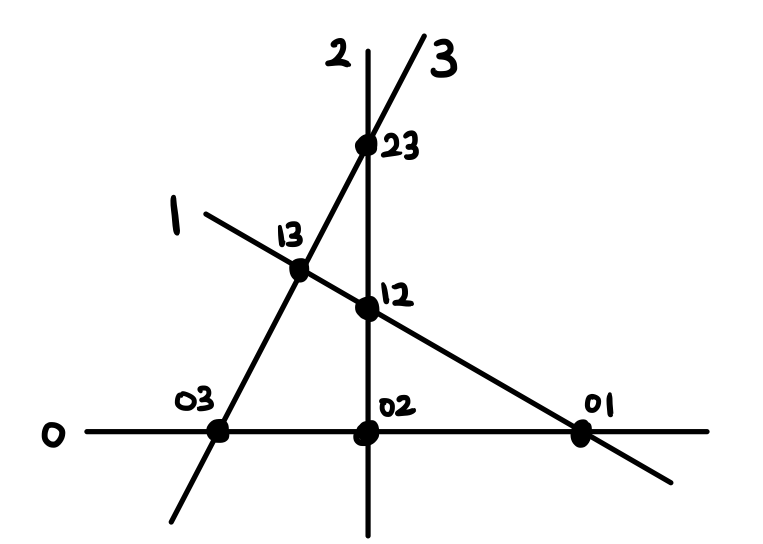}
\end{center}
The wonderful compactification $X_{\mathscr R(M)}$ is given by blowing-up the six points $L_{0,1} , \ldots, L_{2,3}$.  The volume polynomial is
\begin{multline*}
VP_M(\underline t)  = -2{t}_{1}^{2}-2{t}_{3}^{2}-2{t}_{0}^{2}-2{t}_{2}^{2}+2{t}_{3}{t}_{2,3}+2{t}_{2}{t}_{2,3}-{t}_{2,3}^{2}+2{t}_{1}{t}_{1,3}+2{t}_{3}{t}_{1,3}-{t}_{1,3}^{2}+2{t}_{3}{t}_{0,3}+2{t}_{0}{t}_{0,3}\\
 -{t}_{0,3}^{2}+2{t}_{1}{t}_{1,2}+2{t}_{2}{t}_{1,2}-{t}_{1,2}^{2}+2{t}_{0}{t}_{0,2}+2{t}_{2}{t}_{0,2}-{t}_{0,2}^{2}+2{t}_{1}{t}_{0,1}+2{t}_{0}{t}_{0,1}-{t}_{0,1}^{2}.
\end{multline*}

Notice that the coefficient of $t_i^2$ is $-2$ as expected from Corollary \ref{volpol} since $M/\{i\} \simeq U_{2,3}$ so that $|\mu^1(M/\{i\})| = 2$, and likewise the coefficients of $t_{i,j}^2$ are $-1$.  The rest of the coefficients are maximal chains, so the coefficient is ${2\choose 1,1} = 2$.  The shifted rank volume of $M$ is $(4)(-2)(1) +(12)(2)(2) + (6)(-4) = 16 = 4^2$, as expected from Theorem \ref{realmax}.

Let $\pi: X_{\mathscr R(M)} \to \PP^2$ be the blow-down map, $\widetilde H := \pi^* H$ the pullback of the hyperplane class $H\subset \PP^2$, and $E_{ij}$'s the exceptional divisors from the blown-up points.  Then $\operatorname{Pic} X_{\mathscr R(M)} = \ZZ\{\widetilde H, E_{01}, \ldots, E_{23}\}$, where intersection pairing of divisors are $E_{ii'}\cdot E_{jj'} =0 \ \forall \{i,i'\} \neq \{j,j'\}$, $E_{ii'} \cdot \widetilde H = 0$, $\widetilde H\cdot \widetilde H = 1$, and $E_{ii'} \cdot E_{ii'} = -1$.  Hence, $x_0 = \pi^* L_0 = \widetilde H + E_{01} + E_{02} + E_{03}$, so that $x_0^2 = 1 -1 -1 -1 = -2$, as expected.  Similarly, one computes that the shifted rank volume of $M$ is $4^2 = 16$.  Alternatively, note that the map $X_{\mathscr R(M)}\to \PP(H^0(4\widetilde H))$ given by the divisor $4\widetilde H$ factors birationally through $\PP^2$ as the 4-tuple Veronese embedding $\PP^2 \hookrightarrow \PP^{14}$, whose degree is 16.

We remark that the complete linear system $|3\widetilde H - E_{01} - \cdots - E_{23}|$ defines a birational map $X_{\mathscr R(M)}\to \PP^3$ whose image is the Cayley nodal cubic surface (as $(3\widetilde H - E_{01} - \cdots - E_{23})^2 = 3$).  Indeed, as $3\widetilde H - E_{01} - \cdots - E_{23} = 3(x_0 + x_{0,1} + x_{0,2} + x_{0,3} ) - (x_{0,1} + \cdots + x_{2,3}) = 3x_0 + 2x_{0,1}+2x_{0,2}+2x_{0,3} - x_{1,2}-x_{1,3}-x_{2,3}$, evaluating $VP_M$ respectively gives $-2\cdot 3^2 - 3\cdot 2^2 - 3\cdot 1^1 + 3\cdot2\cdot3\cdot2 = 3$.
\end{eg}

\begin{eg}
Let $M = U_{1,1} \oplus U_{2,3}$, another matroid of rank 3 on 4 elements, but not uniform as the above example.  Its lattice of flats $\mathscr L_M$ and the hyperplane arrangement $\mathcal A_{\mathscr R(M)}$ are as follows.
\begin{center}
\begin{tikzpicture}[scale=1, vertices/.style={draw, fill=black, circle, inner sep=0pt}]
                \node [vertices, label=right:{${}$}] (0) at (-0+0,0){};
                \node [vertices, label=right:{${3}$}] (1) at (-2.25+0,1.33333){};
                \node [vertices, label=right:{${2}$}] (2) at (-2.25+1.5,1.33333){};
                \node [vertices, label=right:{${1}$}] (3) at (-2.25+3,1.33333){};
                \node [vertices, label=right:{${0}$}] (4) at (-2.25+4.5,1.33333){};
                \node [vertices, label=right:{${1, 2, 3}$}] (8) at (-2.25+0,2.66667){};
                \node [vertices, label=right:{${0, 3}$}] (5) at (-2.25+1.5,2.66667){};
                \node [vertices, label=right:{${0, 2}$}] (6) at (-2.25+3,2.66667){};
                \node [vertices, label=right:{${0, 1}$}] (7) at (-2.25+4.5,2.66667){};
                \node [vertices, label=right:{${0, 1, 2, 3}$}] (9) at (-0+0,4){};
        \foreach \to/\from in {0/1, 0/2, 0/3, 0/4, 1/8, 1/5, 2/8, 2/6, 3/8, 3/7, 4/5, 4/6, 4/7, 5/9, 6/9, 7/9, 8/9}
        \draw [-] (\to)--(\from);
\end{tikzpicture}
\includegraphics[height=48mm]{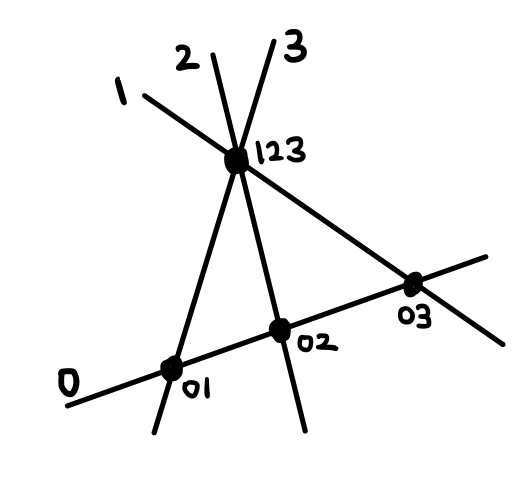}
\end{center}
Its wonderful compactification $\pi: X_{\mathscr R(M)}\to \PP^2$ is the plane $\PP^2$ blown-up at four points $L_{01}, L_{02}, L_{03}, L_{123}$.  The volume polynomial is
\begin{multline*}
VP_M(\underline t)  =-{t}_{{3}}^{2}-{t}_{{2}}^{2}-{t}_{{1}}^{2}-2 {t}_{{0}}^{2}+2{t}_{{3}} {t}_{{0,3}}+2 {t}_{{0}} {t}_{{0,3}}-{t}_{{0,3}}^{2}+2{t}_{{2}} {t}_{{0,2}}+2 {t}_{{0}} {t}_{{0,2}}\\
-{t}_{{0,2}}^{2} +2{t}_{{1}} {t}_{{0,1}}+2 {t}_{{0}} {t}_{{0,1}}-{t}_{{0,1}}^{2}+2{t}_{{3}} {t}_{{1,2,3}}+2 {t}_{{2}} {t}_{{1,2,3}}+2 {t}_{{1}}{t}_{{1,2,3}}-{t}_{{1,2,3}}^{2}.
\end{multline*}
Notice that the coefficient of $t_0^2$ is $-2$ while those of $t_i^2$ ($i\neq 0$) are $-1$, since $\mathscr L_{M/0} \simeq \mathscr L_{U_{2,3}}$ whereas $\mathscr L_{M/1} \simeq \mathscr L_{U_{2,2}}$ (the reduced chromatic polynomials of $U_{2,3}$ and $U_{2,2}$ are $t-2$ and $t-1$).  The shifted rank volume of $M$ is $-5 + (4)(-2^2) + 9(2)(1\cdot 2) = 15 <16$.

Again, let $\widetilde H := \pi^*H$ the pullback of the hyperplane class $H$, and $E_{ij}$ the exceptional divisors of from the blown-up points.  As in the proof of Theorem \ref{realmax}, the shifted rank divisor $D_M$ is $4\widetilde H - E_{123}$, whose volume is $(4\widetilde H - E_{123})^2 = 16- 1 = 15$.  Alternatively, the map $Y_M\to \PP(H^0(4\widetilde H - E_{123}))$ factors birational through $\PP^2$ as a rational map $\PP^2 \to \PP^{13}$ given by a graded linear system $L\subset H^0(\mathscr O_{\PP^2}(4))$ consisting of quartics through the point $H_{123}$.  Its image is the blow-up of a point in $\PP^2$ embedded in $\PP^{13}$ with degree 15.  In summary, we have the commuting diagram
$$\xymatrix{
&X_{\mathscr R(M)} \ar[r] \ar[rd]_\pi &\operatorname{Bl}_{H_{123}}\PP^2 \ar[d] \ar@{^(->}[rrd]\\
& &\PP^2 \ar@{^(->}[r]^{Veronese\quad } &\PP^{{4+2 \choose 2}-1} \ar@{-->}[r] &\PP^{13}.
}$$
\end{eg}

\begin{eg}
We feature one example of rank 4.  Let $M := U_{2,2} \oplus U_{2,3}$.  Its lattice of flats $\mathscr L_M$ and its hyperplane arrangement $\mathcal A_{\mathscr R(M)}$ can be illustrated as follows.

\begin{tabular}{p{10cm} p{5cm}}
\vspace{20pt}
\begin{tikzpicture}[scale=0.9, vertices/.style={draw, fill=black, circle, inner sep=0pt}]
                \node [vertices, label=right:{${}$}] (0) at (-0+0,0){};
                \node [vertices, label=right:{${4}$}] (1) at (-3+0,1.33333){};
                \node [vertices, label=right:{${3}$}] (2) at (-3+1.5,1.33333){};
                \node [vertices, label=right:{${2}$}] (3) at (-3+3,1.33333){};
                \node [vertices, label=right:{${1}$}] (4) at (-3+4.5,1.33333){};
                \node [vertices, label=right:{${0}$}] (5) at (-3+6,1.33333){};
                \node [vertices, label=right:{${4, 2, 3}$}] (13) at (-5.25+0,2.66667){};
                \node [vertices, label=right:{${1, 2}$}] (8) at (-5.25+1.5,2.66667){};
                \node [vertices, label=right:{${4, 1}$}] (6) at (-5.25+3,2.66667){};
                \node [vertices, label=right:{${1, 3}$}] (7) at (-5.25+4.5,2.66667){};
                \node [vertices, label=right:{${0, 4}$}] (9) at (-5.25+6,2.66667){};
                \node [vertices, label=right:{${0, 3}$}] (10) at (-5.25+7.5,2.66667){};
                \node [vertices, label=right:{${0, 2}$}] (11) at (-5.25+9,2.66667){};
                \node [vertices, label=right:{${0, 1}$}] (12) at (-5.25+10.5,2.66667){};
                \node [vertices, label=right:{${1, 2, 3, 4}$}] (17) at (-3+0,4){};
                \node [vertices, label=right:{${0, 2, 3, 4}$}] (18) at (-3+1.5,4){};
                \node [vertices, label=right:{${0, 1, 2}$}] (16) at (-3+3,4){};
                \node [vertices, label=right:{${4, 0, 1}$}] (14) at (-3+4.5,4){};
                \node [vertices, label=right:{${0, 1, 3}$}] (15) at (-3+6,4){};
                \node [vertices, label=right:{${0, 1, 2, 3, 4}$}] (19) at (-0+0,5.33333){};
        \foreach \to/\from in {0/1, 0/2, 0/3, 0/4, 0/5, 1/13, 1/9, 1/6, 2/13, 2/10, 2/7, 3/8, 3/13, 3/11, 4/8, 4/12, 4/6, 4/7, 5/9, 5/10, 5/11, 5/12, 6/17, 6/14, 7/17, 7/15, 8/16, 8/17, 9/14, 9/18, 10/18, 10/15, 11/16, 11/18, 12/16, 12/14, 12/15, 13/17, 13/18, 14/19, 15/19, 16/19, 17/19, 18/19}
        \draw [-] (\to)--(\from);
\end{tikzpicture}

&
\vspace{-10pt}
\includegraphics[height=75mm]{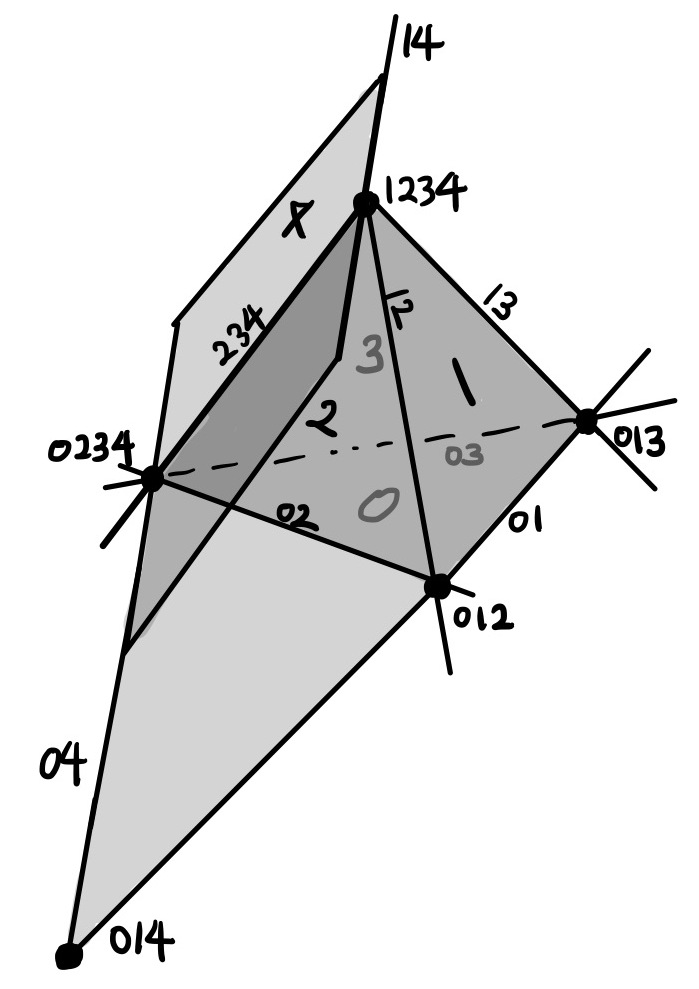}
\end{tabular}

The wonderful compactification $\pi: X_{\mathscr R(M)}\to \PP^3$ is obtained by blowing-up the four points, then the strict transforms of the 8 lines.  The volume polynomial is 
\begin{multline*}
VP_M(\underline t) = {t}_{ {4}}^{3}+{t}_{ {3}}^{3}+{t}_{ {2}}^{3}+2 {t}_{ {1}}^{3}+2 {t}_{ {0}}^{3}-3 {t}_{ {4}} {t}_{ {4,1}}^{2}-3 {t}_{ {1}} {t}_{ {4,1}}^{2}+2 {t}_{ {4,1}}^{3}-3 {t}_{ {3}} {t}_{ {1,3}}^{2}-3 {t}_{ {1}} {t}_{ {1,3}}^{2}+2 {t}_{ {1,3}}^{3}-3 {t}_{ {2}} {t}_{ {1,2}}^{2}\\
-3 {t}_{ {1}} {t}_{ {1,2}}^{2}+2 {t}_{ {1,2}}^{3}-3 {t}_{ {4}} {t}_{ {0,4}}^{2}-3 {t}_{ {0}} {t}_{ {0,4}}^{2}+2 {t}_{ {0,4}}^{3}-3 {t}_{ {3}} {t}_{ {0,3}}^{2}-3 {t}_{ {0}} {t}_{ {0,3}}^{2}+2 {t}_{ {0,3}}^{3}-3 {t}_{ {2}} {t}_{ {0,2}}^{2}-3 {t}_{ {0}} {t}_{ {0,2}}^{2}\\
+2 {t}_{ {0,2}}^{3}-6 {t}_{ {1}} {t}_{ {0,1}}^{2}-6 {t}_{ {0}} {t}_{ {0,1}}^{2}+4 {t}_{ {0,1}}^{3}-3 {t}_{ {4}} {t}_{ {4,2,3}}^{2}-3 {t}_{ {3}} {t}_{ {4,2,3}}^{2}-3 {t}_{ {2}} {t}_{ {4,2,3}}^{2}+2 {t}_{ {4,2,3}}^{3}-3 {t}_{ {4}}^{2} {t}_{ {4,0,1}}\\
-3 {t}_{ {1}}^{2} {t}_{ {4,0,1}} -3 {t}_{ {0}}^{2} {t}_{ {4,0,1}}+6 {t}_{ {4}} {t}_{ {4,1}} {t}_{ {4,0,1}}+6 {t}_{ {1}} {t}_{ {4,1}} {t}_{ {4,0,1}}-3 {t}_{ {4,1}}^{2} {t}_{ {4,0,1}}+6 {t}_{ {4}} {t}_{ {0,4}} {t}_{ {4,0,1}}+6 {t}_{ {0}} {t}_{ {0,4}} {t}_{ {4,0,1}}\\
-3 {t}_{ {0,4}}^{2} {t}_{ {4,0,1}}+6 {t}_{ {1}} {t}_{ {0,1}} {t}_{ {4,0,1}}+6 {t}_{ {0}} {t}_{ {0,1}} {t}_{ {4,0,1}}-3 {t}_{ {0,1}}^{2} {t}_{ {4,0,1}}-3 {t}_{ {4}} {t}_{ {4,0,1}}^{2}-3 {t}_{ {1}} {t}_{ {4,0,1}}^{2}-3 {t}_{ {0}} {t}_{ {4,0,1}}^{2}\\
+{t}_{ {4,0,1}}^{3}-3 {t}_{ {3}}^{2} {t}_{ {0,1,3}}-3 {t}_{ {1}}^{2} {t}_{ {0,1,3}}-3 {t}_{ {0}}^{2} {t}_{ {0,1,3}}+6 {t}_{ {3}} {t}_{ {1,3}} {t}_{ {0,1,3}}+6 {t}_{ {1}} {t}_{ {1,3}} {t}_{ {0,1,3}}-3 {t}_{ {1,3}}^{2} {t}_{ {0,1,3}}\\
+6 {t}_{ {3}} {t}_{ {0,3}} {t}_{ {0,1,3}}+6 {t}_{ {0}} {t}_{ {0,3}} {t}_{ {0,1,3}}-3 {t}_{ {0,3}}^{2} {t}_{ {0,1,3}}+6 {t}_{ {1}} {t}_{ {0,1}} {t}_{ {0,1,3}}+6 {t}_{ {0}} {t}_{ {0,1}} {t}_{ {0,1,3}}-3 {t}_{ {0,1}}^{2} {t}_{ {0,1,3}}-3 {t}_{ {3}} {t}_{ {0,1,3}}^{2}\\
-3 {t}_{ {1}} {t}_{ {0,1,3}}^{2}-3 {t}_{ {0}} {t}_{ {0,1,3}}^{2}+{t}_{ {0,1,3}}^{3}-3 {t}_{ {2}}^{2} {t}_{ {0,1,2}}-3 {t}_{ {1}}^{2} {t}_{ {0,1,2}}-3 {t}_{ {0}}^{2} {t}_{ {0,1,2}}+6 {t}_{ {2}} {t}_{ {1,2}} {t}_{ {0,1,2}}+6 {t}_{ {1}} {t}_{ {1,2}} {t}_{ {0,1,2}}\\
-3 {t}_{ {1,2}}^{2} {t}_{ {0,1,2}}+6 {t}_{ {2}} {t}_{ {0,2}} {t}_{ {0,1,2}}+6 {t}_{ {0}} {t}_{ {0,2}} {t}_{ {0,1,2}}-3 {t}_{ {0,2}}^{2} {t}_{ {0,1,2}}+6 {t}_{ {1}} {t}_{ {0,1}} {t}_{ {0,1,2}}+6 {t}_{ {0}} {t}_{ {0,1}} {t}_{ {0,1,2}}\\
-3 {t}_{ {0,1}}^{2} {t}_{ {0,1,2}}-3 {t}_{ {2}} {t}_{ {0,1,2}}^{2}-3 {t}_{ {1}} {t}_{ {0,1,2}}^{2}-3 {t}_{ {0}} {t}_{ {0,1,2}}^{2}+{t}_{ {0,1,2}}^{3}-3 {t}_{ {4}}^{2} {t}_{ {1,2,3,4}}-3 {t}_{ {3}}^{2} {t}_{ {1,2,3,4}}-3 {t}_{ {2}}^{2} {t}_{ {1,2,3,4}}\\
-6 {t}_{ {1}}^{2} {t}_{ {1,2,3,4}}+6 {t}_{ {4}} {t}_{ {4,1}} {t}_{ {1,2,3,4}}+6 {t}_{ {1}} {t}_{ {4,1}} {t}_{ {1,2,3,4}}-3 {t}_{ {4,1}}^{2} {t}_{ {1,2,3,4}}+6 {t}_{ {3}} {t}_{ {1,3}} {t}_{ {1,2,3,4}}+6 {t}_{ {1}} {t}_{ {1,3}} {t}_{ {1,2,3,4}}\\
-3 {t}_{ {1,3}}^{2} {t}_{ {1,2,3,4}}+6 {t}_{ {2}} {t}_{ {1,2}} {t}_{ {1,2,3,4}}+6 {t}_{ {1}} {t}_{ {1,2}} {t}_{ {1,2,3,4}}-3 {t}_{ {1,2}}^{2} {t}_{ {1,2,3,4}}+6 {t}_{ {4}} {t}_{ {4,2,3}} {t}_{ {1,2,3,4}}+6 {t}_{ {3}} {t}_{ {4,2,3}} {t}_{ {1,2,3,4}}\\
+6 {t}_{ {2}} {t}_{ {4,2,3}} {t}_{ {1,2,3,4}}-3 {t}_{ {4,2,3}}^{2} {t}_{ {1,2,3,4}}-3 {t}_{ {4}} {t}_{ {1,2,3,4}}^{2}-3 {t}_{ {3}} {t}_{ {1,2,3,4}}^{2}-3 {t}_{ {2}} {t}_{ {1,2,3,4}}^{2}-3 {t}_{ {1}} {t}_{ {1,2,3,4}}^{2}+{t}_{ {1,2,3,4}}^{3}\\
-3 {t}_{ {4}}^{2} {t}_{ {0,2,3,4}}-3 {t}_{ {3}}^{2} {t}_{ {0,2,3,4}}-3 {t}_{ {2}}^{2} {t}_{ {0,2,3,4}}-6 {t}_{ {0}}^{2} {t}_{ {0,2,3,4}}+6 {t}_{ {4}} {t}_{ {0,4}} {t}_{ {0,2,3,4}}+6 {t}_{ {0}} {t}_{ {0,4}} {t}_{ {0,2,3,4}}\\
-3 {t}_{ {0,4}}^{2} {t}_{ {0,2,3,4}}+6 {t}_{ {3}} {t}_{ {0,3}} {t}_{ {0,2,3,4}}+6 {t}_{ {0}} {t}_{ {0,3}} {t}_{ {0,2,3,4}}-3 {t}_{ {0,3}}^{2} {t}_{ {0,2,3,4}}+6 {t}_{ {2}} {t}_{ {0,2}} {t}_{ {0,2,3,4}}+6 {t}_{ {0}} {t}_{ {0,2}} {t}_{ {0,2,3,4}}\\
-3 {t}_{ {0,2}}^{2} {t}_{ {0,2,3,4}}+6 {t}_{ {4}} {t}_{ {4,2,3}} {t}_{ {0,2,3,4}}+6 {t}_{ {3}} {t}_{ {4,2,3}} {t}_{ {0,2,3,4}}+6 {t}_{ {2}} {t}_{ {4,2,3}} {t}_{ {0,2,3,4}}-3 {t}_{ {4,2,3}}^{2} {t}_{ {0,2,3,4}}-3 {t}_{ {4}} {t}_{ {0,2,3,4}}^{2}\\
-3 {t}_{ {3}} {t}_{ {0,2,3,4}}^{2}-3 {t}_{ {2}} {t}_{ {0,2,3,4}}^{2}-3 {t}_{ {0}} {t}_{ {0,2,3,4}}^{2}+{t}_{ {0,2,3,4}}^{3}.\\
\end{multline*}
Its shifted rank volume is 112, which is the smallest for simple matroids of rank 4 on 5 elements.  Let $\widetilde H = \pi^*H$ be the pullback of hyperplane $H\subset \PP^3$ again, and let $E_F$'s be the exceptional divisors from the blow-ups.  The shifted rank divisor $D_M$ is $5\widetilde H - E_{0234} - E_{1234} - E_{234}$.  Denote by $P,Q$ the two points $L_{0234}, L_{1234}$, and $\ell$ the line $L_{234}$.  Then the map given by $D_M$ on $X_{\mathscr R(M)}$ factors through a rational map on $\PP^3$ as follows.  First, consider the rational map given by a linear series $L\subset H^0(\mathscr O_{\PP^3}(5))$ consisting of quintic hypersurfaces through the two points $P,Q$.  Then, consider the map from $\operatorname{Bl}_{P,Q}\PP^3$ given by divisors in $H^0(\operatorname{Bl}_{P,Q}\PP^3)$ containing the strict transform of the line $\ell$.  The image of the map is the blow-up of $\ell$ in $\PP^3$ embedded in $\PP^{49}$ with degree 112.  In summary, we have
$$\xymatrix{
&Y_M \ar[r] \ar[rdd]_\pi &\operatorname{Bl}_L(\operatorname{Bl}_{P,Q}\PP^3) \ar[d]\ar[rrd]\\
& &\operatorname{Bl}_{P,Q}\PP^3 \ar[d] \ar@/^/[rrd] & &\operatorname{Bl}_L\PP^3 \ar@{^(->}[rd]\\
& &\PP^3 \ar@{^(->}[r]^{Veronese\quad } &\PP^{{5+3\choose 3}-1} \ar@{-->}[r] &\PP^{53}  \ar@{-->}[r] &\PP^{49}.
}$$
\end{eg}


\bigskip
\noindent\textbf{Acknowledgements.} The author would like to thank Bernd Sturmfels for suggesting this problem, Justin Chen for many helpful discussions and his Macaulay2 package on matroids \cite{Che15}, Alex Fink for generously providing the sketch of the proof for valuativeness of the volume polynomial, and Federico Ardila for inspiring the author to pursue this direction of research.  The author is also grateful for helpful conversations with David Eisenbud, June Huh, Vic Reiner, and Mengyuan Zhang, and thanks the two references for their careful reading.  The author acknowledges the support of James H. Simons Fellowship the spring of 2018.

\end{document}